\documentclass[a4paper,12pt]{article}

\usepackage{pstricks}
\usepackage{graphicx,psfrag}
\usepackage{amsmath}
\usepackage{amssymb}
\usepackage{amsthm}

\usepackage{color}

      \newcommand {\al}   {\alpha}          
      \newcommand {\gam } {\gamma}          
      \newcommand {\del}  {\delta}          
              \newcommand {\ve}   {\varepsilon}
                 \newcommand {\vphi} {\varphi}
      \newcommand {\lam}  {\lambda}         \newcommand {\Lam}  {\Lambda}
                
                \newcommand {\Om}  {\Omega}
      \newcommand {\pl}   {\partial}        
           
           \newcommand {\UUU}  {{\cal U}}

                 \newcommand {\RRR}  {{\mathbb R}}

      \newcommand {\FFF}  {{\cal F}}        
           \newcommand {\NNN}  {\mathbb{N}}
       \newcommand {\JJJ}  {J}

      \newcommand {\xxx}   {\upsilon_1}          \newcommand {\yyy}  {\upsilon_2}
      \newcommand {\xxyy}   {\upsilon}           \newcommand {\KKK}   {K}

     \newcommand {\beq}  {\begin{equation}}
      \newcommand {\eeq}  {\end{equation}}  
     \newcommand {\beqo}  {\begin{equation*}}
      \newcommand {\eeqo}  {\end{equation*}}

      \newtheorem{theorem}{Theorem}
      \newtheorem{lemma}{Lemma}
      
      \newtheorem{zam}{Remark}
      \newtheorem{opr}{Definition}
      \newtheorem{corollary}{Corollary}
      \newtheorem{primer}{Example}

\author{Alexander Plakhov\thanks{Center for R\&{}D in Mathematics and Applications, Department of Mathematics, University of Aveiro, Portugal and Institute for Information Transmission Problems, Moscow, Russia.}}

\title{Method of nose stretching in Newton's \\ problem of minimal resistance} 

\begin{document}

\maketitle

\begin{abstract}
We consider the problem $\inf\big\{ \int\!\!\int_\Omega (1 + |\nabla u(x,y)|^2)^{-1} dx dy : \text{ the function } u : \Omega \to \mathbb{R} \text{ is concave and } 0 \le u(x,y) \le M \text{ for all } (x,y) \in \Omega =\{ (x,y): x^2 + y^2 \le 1 \} \, \big\}$ (Newton's problem) and its generalizations.
In the paper \cite{BrFK} it is proved that if a solution $u$ is $C^2$ in an open set $\mathcal{U} \subset \Omega$ then $\det D^2u = 0$ in $\mathcal{U}$. It follows that graph$(u)\rfloor_\mathcal{U}$ does not contain extreme points of the subgraph of $u$.

In this paper we prove a somewhat stronger result. Namely, there exists a solution $u$ possessing the following property. If $u$ is $C^1$ in an open set $\mathcal{U} \subset \Omega$ then graph$(u\rfloor_\mathcal{U})$ does not contain extreme points of the convex body $C_u = \{ (x,y,z) :\, (x,y) \in \Omega,\ 0 \le z \le u(x,y) \}$. As a consequence, we have $C_u = \text{\rm Conv} (\overline{\text{\rm Sing$C_u$}})$, where Sing$C_u$ denotes the set of singular points of $\partial C_u$. We prove a similar result for a generalized Newton's problem.
\end{abstract}

\begin{quote}
\end{quote}

\begin{quote} {\small {\bf Key words and phrases:}
Convex bodies, Newton's problem of minimal resistance, surface area measure, Blaschke addition, the method of nose stretching.}
\end{quote}

\section{Introduction}

\subsection{History of the problem}

Isaac Newton in his {\it Principia} (1687) considered the following mechanical model. A solid body moves with constant velocity through a rarefied medium composed of point particles. The particles are initially at rest, and all collisions of particles with the body's surface are perfectly elastic.\footnote{Actually, Newton considered a one-parameter family of laws of reflection. Namely, a parameter $0 \le k \le 1$ is fixed, and in a reference system connected with the body, at each impact, the normal component of the particle's velocity of incidence is multiplied by $-k$, while the tangential component remains unchanged. In the case $k = 1$ we have the law of perfectly elastic (billiard) reflection.}

The medium is assumed to be extremely rarefied, so as mutual interaction of the medium particles can be neglected. In physical terms, this means that the free path length of particles is much larger than the size of the body. As a real-world application, one can imagine an artificial satellite with well polished surface moving around the Earth at low altitudes (between 100 and 1000 km) when the atmosphere is extremely thin (but is still present).

As a result of collisions of the body with the particles, the force of resistance is created, which acts on the body and slows down its velocity. Newton calculated the resistance of several geometrical shapes (a cylinder, a cone, a sphere) and considered the following problem: minimize the resistance in the class of convex bodies that are rotationally symmetric with respect to a line parallel to the direction of motion and have fixed length along the direction of motion and fixed maximal width in the orthogonal direction.

Newton described the solution of this problem in geometrical terms (see, e.\,g., the book \cite{TBook}). The solution looks like a truncated cone with slightly inflated lateral surface; see Fig.~\ref{figNew} for the case when the length is equal to the maximal width.
        \begin{figure}[h]
\centering
\includegraphics[scale=0.45]{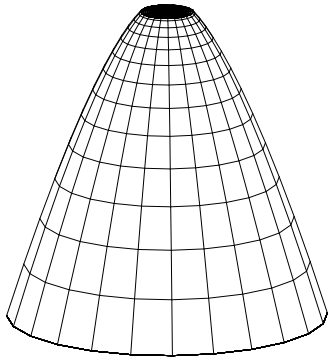}
\caption{A solution to the rotationally symmetric Newton problem.}
\label{figNew}
\end{figure}
The body in the picture moves in the medium vertically upward;  equivalently, one can assume that the body is at rest and there is an incident flow of particles moving vertically downward.


In modern terms the problem can be stated as follows. Choose a reference system with the coordinates $x$, $y$, $z$ connected with the body so as the $z$-axis coincides with the body's axis of symmetry and is counter directional to the flow of particles. Let the front part of the body be the graph of a radial function $z = \vphi(r)$, $r = \sqrt{x^2 + y^2}$, $0 \le r \le L$. Since the body is convex, the function $\vphi$ is concave and monotone non-increasing.

The resistance equals $2\pi\rho v^2 R(\vphi)$, where $\rho$ is the density of the medium, $v$ is the scalar velocity of the body, and
$$
R(\vphi) = \int_0^L \frac{1}{1 + \vphi'(r)^2}\, r\, dr.
$$
The values $\rho$ and $v$ are fixed, so the problems is as follows: minimize $R(\vphi)$ in the class of concave monotone non-increasing functions $\vphi : [0,\, L] \to \RRR$ such that $0 \le \vphi(r) \le M$. Here $M$ and $2L$ are, respectively, the fixed length and maximal width of the body.

The new life was given to the problem in 1993 when the paper by Buttazzo and Kawohl \cite{BK} was published. The authors considered the problem of minimal resistance in various classes, in particular in the class of convex (not necessarily symmetric) bodies and in the class of nonconvex bodies satisfying the so-called single impact condition. Since then, many research papers have been published in this area.

In the papers \cite{BFK,BrFK,BG97,CLR03,LP1,LZ}, the problem in various classes of convex bodies have been studied. In particular, in \cite{LP1,LZ} there are considered classes of bodies that are convex hulls of a certain pair of planar convex curves. In \cite{BK,BFK,MMOP}, the problem in some classes of (generally) nonconvex bodies are considered. In \cite{surveyK,surveyB}, surveys of the current state of the problem are given.

The problem for rotationally symmetric bodies is studied under the additional conditions that the so-called arclength is fixed \cite{BW}; there is a friction in the course of body-particle interaction \cite{friction}; the thermal motion of the medium particles is present \cite{temp}.

In \cite{CL1,CL2,CL3,SIREV,AkopPl,Nonl2016}, the problem is studied in classes of nonconvex bodies satisfying the condition that each particle of the flow hits the body only once. Connection of this problem with Besicovitch's solution of the Kakeya needle problem \cite{Bes} have been found in \cite{SIREV,Nonl2016}.

In \cite{Canadian}, \cite{bookP} the techniques of the theory of billiards are used to study the problem in classes of nonconvex bodies with multiple body-particle collisions allowed. The problems of resistance optimization for bodies that perform both translational and rotational motion are studied in \cite{PG,ARMA,SIMArough,PTG,OMT}. In these studies, methods of optimal mass transport theory \cite{Villani,MC} are used, and applications to geometrical optics and mechanics (invisibility \cite{0-resist,bookP,invisibility}, Magnus effect \cite{PTG}, retroreflectors \cite{tube,retro}, camouflaging \cite{camouflage}) are found.

Suppose that the body is convex and a reference system connected with the body with the coordinates $x$, $y$, $z$ is chosen so as the $z$-axis is parallel and counter directional to the direction of the flow. Let the front part of the body's surface be the graph of a concave function $u : \Om \to \RRR$, where $\Om \subset \RRR^2$ is the projection of the body on the $xy$-plane. Then the vertical component of resistance is equal to $2\rho v^2 \FFF[u]$, where $\FFF[u]$ is given by the formula
\beq\label{problem classical}
\FFF[u] = \int\!\!\!\int_\Om \frac{1}{1 + |\nabla u(x_1,x_2)|^2}\, dx_1\, dx_2.
\eeq
In what follows we assume that the numerical coefficient $2\rho v^2$ equals 1. The value in \eqref{problem classical} will be called {\it resistance}.

The earliest and the most direct, and perhaps the most difficult generalization of the original problem stated by Newton is as follows: find the body of minimal resistance in the class of convex bodies with fixed length along the direction of motion and fixed projection on the plane orthogonal to this direction. The only difference from the original problem is that the body is in general not rotationally symmetric (as a consequence, the orthogonal projection of the body is not necessarily a ball). In terms of the function $u$, the problem reads as follows \cite{BK}:
\vspace{1mm}

{\bf Generalized Newton's problem.} {\it Given $M > 0$ and a convex body\footnote{A convex body is a compact convex set with nonempty interior.} $\Om \subset \RRR^2$, find the minimum of the functional \eqref{problem classical} in the class of concave functions $u : \Om \to \RRR$ satisfying the condition $0 \le u(x_1,x_2) \le M$ for all $(x_1,x_2) \in \Om$.}
\vspace{1mm}

The following is known about a solution $u$ to this problem.

{\bf P$\mathbf{_1}$.} There exists at least one solution \cite{BFK}.

{\bf P$\mathbf{_2}$.} Let $\nabla u$ exist at a point $(x,y) \in \Om$; then either $|\nabla u(x,y)| \ge 1$ or $|\nabla u(x,y)| = 1$ \cite{BFK}.

{\bf P$\mathbf{_3}$.} If $u$ is $C^2$ in a neighborhood of $(x,y)$ and $0 < u(x,y) < M$ then the matrix of the second derivatives $D^2u(x,y)$ has a zero eigenvalue.

{\bf P$\mathbf{_4}$.} The top level set $\{ u(x,y) = M \}$ is neither empty nor a singleton. Therefore, it is either a line segment or a planar convex body (G. Buttazzo, personal communication).

{\bf P$\mathbf{_5}$.} If all points of the curve $\pl\Om$ are regular (for example, $\Om$ is a circle) then $u\rfloor_{\pl\Om} = 0$ \cite{boundary}.

Numerical study of the problem in the case when $\Om$ is a circle has been carried out in the papers \cite{LO,W}.   


Numerical results suggest that the graph of a solution $u$ is a piecewise developable surface; more precisely, singular points of the graph form several curves, and the graph is the union of line segments with the endpoints on these curves. Property P$_3$ is an argument in favor of this conjecture. Unfortunately, $u$ may not be $C^2$ in any open set.

In the present paper we prove that if the curve $\pl C$ is regular, then the graph of a certain solution is formed by line segments with the endpoints in the closure of the set of singular points of the graph (Corollary \ref{cor3}). It means that the solution is defined uniquely by the set of singular points, and therefore, to study the properties of the solution, it suffices to concentrate on its singular points.

Our result is equivalent to property P$_3$ with $C^2$ substituted by $C^1$. Unfortunately, we still cannot guarantee that the set of singular points is closed. We cannot even affirm that this set is not dense in the graph. We believe that the following conjecture is true.
\vspace{1mm}

{\bf Conjecture.} {\it The set of singular points of a solution is closed, and therefore, is nowhere dense.}
\vspace{1mm}

The main result of this paper is Theorem \ref{T main} (subsection \ref{subs12}) stated for a further generalization of Newton's problem and proved in Section \ref{sec Tmain}, and Corollary \ref{cor3} is its simple consequence.
In general, we believe that it is fruitful to work with a generalized version of the problem, and that further progress can be achieved by using methods of convex geometry, including the notion of surface area measure.

\subsection{Statement of the results and discussion}\label{subs12}

Let us imagine again an artificial satellite at a low-Earth orbit. It is known that the thermal velocity of molecules in the  atmosphere is comparable with the satellite's velocity and therefore cannot be neglected. Moreover, the interaction of molecules with the satellite's surface by no means obeys the law of elastic reflection. This implies that formula \eqref{problem} for resistance may not be valid. Besides, not only the front (exposed to the flow) part of the satellite's surface, but also its rear part should be taken in consideration.

In order to deal with the problem in this general setting, it is more natural and convenient to work with convex bodies, rather than with concave functions.

First introduce the notation. A {\it convex body}  is a compact convex set with nonempty interior. In this paper, the letter $C$ (also with some subscripts or superscripts) is always used to designate a convex body in $\RRR^3$. A point $\xi \in \pl C$ is called {\it singular} if there is more than one plane of support to $C$ at $\xi$, and {\it regular} otherwise. The set of singular points of $\pl C$ is denoted by Sing$C$. It is known that almost all points of $\pl C$ are regular. The outward unit normal to $C$ at a regular point $\xi \in \pl C$ is denoted by $n_\xi$. If a plane of support at $\xi$ is unique (and therefore $\xi$ is regular), it is called the {\it tangent plane} at $\xi$.

A point $x \in C$ is an {\it extreme point} of $C$, if it is not a convex combination $x = \lam a + (1-\lam) b$,\, $a,\, b \in C$,\, $a \ne b$,\, $0 < \lam < 1$. The set of extreme points of $C$ is denoted as Ext$C$. The convex hull of a set $A$ is denoted as Conv$A$. A plane of support to a convex body is always assumed to be oriented by the outward normal vector.

\begin{opr}
Let $D \subset \pl C$ be a Borel set. The {\rm surface area measure} of $D$ is the measure $\nu_D$ in $S^2$ defined by
$$
\nu_D(A) = \text{Leb}\big( \{ \xi \in D : \ n_\xi \in A \} \big)
$$
for any Borel set $A \subset S^2.$ Here Leb means the standard Lebesgue measure (area) on $\pl C$. In the particular case when $D = \pl C$, the corresponding surface area measure is denoted as $\nu_{\pl C}$.
\end{opr}

If $D$ is not a planar set, the measure $\nu_D$ does not depend on the choice of the convex body $C$ whose boundary contains $D$. If $D$ is planar, the measure depends on the choice of the normal to $D$ ($n$ or $-n$); in what follows it will always be clear, which normal should be chosen.

Let $f : S^2 \to \RRR$ be a continuous function and let $D \subset \pl C$ be a Borel subset of a convex body $C \subset \RRR^3$. We define the functional
\beqo\label{Functional}
F(D) = \int_{D} f(n_\xi)\, d\xi,
\eeqo
where $d\xi$ denotes the standard 2-dimensional Lebesgue measure on $\pl C$. Making a change of variable, one can write this functional as
$$F(D) = \int_{S^2} f(n)\, d\nu_D(n).$$
Again, if $D$ is not planar, the value $F(D)$ does not depend on $C$, and if $D$ is planar, the choice of the normal, and therefore the value $F(D)$, will always be uniquely defined.

Let two convex sets $C_1 \subset C_2 \subset \RRR^3$ be fixed. We consider the problem: 
\beq\label{problem}
\text{Minimize }  \, F(\pl C) = \int_{\pl C} f(n_\xi)\, d\xi \, \text{ in the class of convex bodies $C_1 \subset C \subset C_2$.}
\eeq

It is known that for any continuous function $f$ and each pair of convex sets $C_1 \subset C_2$, this problem has at least one solution; see the paper \cite{BG97}.

This statement of the problem can be interpreted as follows. In the real-life case, the pressure of the flow at a point $\xi \in \pl C$ depends only on the slope of the surface at $\xi$, that is, equals $p(n_\xi)\, n_\xi$, where $p$ is a continuous function defined on $S^2$. The projection of the drag force on the $z$-axis equals $F(\pl C) = \int_{\pl C} f(n_\xi)\, d\xi$, where $f(n) = p(n)\, n_3$ for $n = (n_1, n_2, n_3) \in S^2$.

The condition $C_1 \subset C \subset C_2$ can also be reasonably interpreted. Suppose we are given a metal body occupying the domain $C_2 \setminus C_1$ and are going to remove a part of material of the body to produce the optimal streamlined shape when moving in a certain direction. The resulting shape is $C \setminus C_1$, where $C$ satisfies the above condition.

\begin{primer}
Consider the cylinder $C_2 = \Om \times [0,\, M]$ and its bottom $C_1 = \Om \times \{ 0 \}$, where $\Om \in \RRR^2$ is a convex body and $M > 0$. Then each body $C$ satisfying the condition $C_1 \subset C \subset C_2$ is bounded below by $C_1$ and above by the graph of a concave function $u : \Om \to \RRR$ satisfying $0 \le u \le M$, that is,
$$
C = C_u = \{ (x,y,z) :\, (x,y) \in \Om,\ 0 \le z \le u(x,y) \}.
$$
The body's boundary $\pl C$ is the union of the disc $C_1$, a part of the cylindrical boundary $\pl \Om \times [0,\, M]$, and the graph of $u.$
Let $f(n) = p(n)\, n_3$, where $p$ is a function on $S^2$. Then the integral $\int f(n_\xi)\, d\xi$ over the cylindrical boundary of $C$ is zero, the integral over the disc $C_1$ is constant, and the integral over the graph of $u$ (after the change of variable $\xi \rightsquigarrow x_1, x_2$, taking into account that $d\xi = \sqrt{1 + u_x'^2 + u_y'^2}\, dx_1 dx_2$) equals
\beqo\label{pru}
\FFF[u] = \int\!\!\!\int_\Om g(\nabla u(x_1,x_2))\, dx_1 dx_2,
\eeqo
where
$$
g(\xxx,\yyy) = p\Big( \frac{1}{\sqrt{1+\xxx^2+\yyy^2}}(-\xxx, -\yyy, 1) \Big).
$$
Therefore problem \eqref{problem} is equivalent to the problem of minimization of the functional $\FFF$ in the class of concave functions $u : \Om \to \RRR$ satisfying the condition $0 \le u(x_1,x_2) \le M$ for all $(x_1,x_2)\in \Om$.
\end{primer}

\begin{primer}
Generalized Newton's problem of minimization of the functional in \eqref{problem classical} corresponds to the case when $f(n) = (n_3)_+^3 = p(n) n_3$ with $p(n) = (n_3)_+^2$,\,$C_2$ is the cylinder and $C_1$ is its bottom, $C_2 = \Om \times [0,\, M]$,\, $C_1 = \Om \times \{ 0 \}$. Here $\Om$ is the unit circle, $z_+ = \max \{ 0,\, z \}$ means the positive part of $z$, and $M > 0$ is the parameter of the problem. This functional is relevant in  the model where the absolute temperature of the medium is 0 and the body-particle collisions are perfectly elastic (billiard-like).
\end{primer}

\begin{zam}
In general the solution to problem \eqref{problem} may not be unique. For example, suppose that $C_1 = \Om \times \{ 0 \}$,\, $C_2 = \Om \times [0,\, 1]$, $f > 0$ in a small neighborhood of $(0,0,1)$, and $f = 0$ outside this neighborhood. Then the minimal value of the functional is 0 and it is attained at a family of bodies of the form Conv$(C_1 \cup \{ (a,b,c) \})$, $(a,b) \in \Om$, with $c \le 1$ being sufficiently close to 1.
\end{zam}

\begin{zam}\label{zam numerics}
In generalized Newton's problem \eqref{problem classical} with the circular base, $\Om = \{ x_1^2 + x_2^2 \le 1 \}$, the numerical study \cite{LO,W} seems to indicate that there exists a sequence of values $+\infty = M_1 > M_2 > M_3 > \ldots$ converging to zero such that for $M_k < M < M_{k-1},\, k = 2,\, 3,\ldots$ the solution is unique (up to a rotation about the $z$-axis)  and the top level set $\{ u(x,y) = M \}$ is a regular $k$-gon, and for each value $M = M_k$ there are two distinct solutions with the top level sets being a regular $k$-gon and a regular $(k+1)$-gon.
\end{zam}

Let $g : \RRR^2 \to \RRR$ be a continuous function, $\Om \subset \RRR^2$ be a convex body, and $M$ be a positive value. Consider the following problem:
\beq\label{f1}
\text{ Minimize the functional} \quad \FFF[u] = \int\!\!\!\int_\Om g(\nabla u(x_1,x_2))\, dx_1 dx_2
\eeq
in the class of concave functions $u : \Om \to \RRR$ satisfying the condition $0 \le u(x) \le M$ for all $x = (x_1, x_2) \in \Om$.

Denote by $D^2u(x)$ the matrix of second derivatives (whenever it exists),
$$D^2u(x) = \left[\! \begin{array}{cc}
u''_{x_1x_1}(x) & u''_{x_1x_2}(x)\\ u''_{x_2x_1}(x) & u''_{x_2x_2}(x)
\end{array} \!\right].$$

The following statement holds true.

\begin{theorem}\label{T brfk}
Let $g$ be of class $C^2$ and the matrix $D^2g$ have at least one negative eigenvalue for all values of the argument. Assume that $u$ is a solution of problem \eqref{f1} and there is an open set $\UUU \subset \Om$ such that

(a) $u \in  C^2(\UUU)$,

(b) $u(x) < M$  for all $x \in \UUU$.\\
Then $\det D^2u(x) = 0$ for all $x \in \UUU$.
\end{theorem}

This theorem was proved in \cite{BrFK} (Theorem 2.1 and Remark 3.4) in the particular case when $g(\xxyy) = 1/(1+|\xxyy|^2)$. In the general case the proof is basically the same. For the reader's convenience, it is provided in Section \ref{sec Tbrfk}.

The statement of Theorem \ref{T brfk} implies that the gaussian curvature at each point of the surface $\{ (x, u(x)) : x \in \UUU \}$ equals zero, and therefore, the surface is developable. It follows that no point of the surface is an extreme point of the body
$$
C_u = \{ (x,y,z) :\, (x,y) \in \Om,\ 0 \le z \le u(x,y) \}.
$$
In other words, we have the following

\begin{corollary}\label{cor1}
Under the assumptions of Theorem \ref{T brfk} we have
$$
\text{\rm Ext} C_u \cap \{ (x, u(x)) : x \in \UUU \} = \emptyset.
$$
\end{corollary}

The following question still remains. Suppose that an open subset of the lateral boundary of an optimal body does not contain singular points. Is it true that it does not contain extreme points (and therefore, is developable)?

Note that in this question only $C^1$ (rather than $C^2$) smoothness of the surface is a priori assumed.

We shall prove that the answer to this question is positive, even in the case of more general problem \eqref{problem}, for at least one solution of the problem.     

The main result of this paper is the following theorem.

\begin{theorem}\label{T main}
There is a solution $\widehat C$ to problem \eqref{problem} such that
\beq\label{eq property}
\text{\rm Ext}\widehat C \subset \pl C_1 \cup \pl C_2 \cup \overline{\text{\rm Sing}\widehat C}.
\eeq
Here and in what follows, the bar means closure.

Moreover, if a solution $C$ does not satisfy \eqref{eq property} then the set of solutions is extremely degenerate; namely, there is a family of solutions $\{ C(\vec s): \ \vec s = (s_i)_{i\in\mathbb{N}} \in [0,\, 1]^\infty \}$, where $\vec 0 = (0,\, 0, \ldots)$, $\vec 1 = (1,\, 1, \ldots)$, such that $C(\vec 0) = C$ and $\widehat C = C(\vec 1)$ satisfies \eqref{eq property}. The corresponding family of surface area measures is linear and infinite dimensional; that is, 
there is a linearly independent set of signed measures $\nu_i,\, i \in \mathbb{N}$ such that for all $\vec s \in [0,\, 1]^\infty$,
$$
\nu_{\pl C(\vec s)} = \nu_{\pl C} + \sum_i s_i \nu_i.
$$
\end{theorem}

This theorem will be proved in Section \ref{sec Tmain}.

Using the Krein-Milman theorem, one immediately obtains the following corollary of Theorem \ref{T main}.

\begin{corollary}\label{cor2}
There is at least one solution $C$ to problem \eqref{problem} satisfying the equality
$$
C = \text{\rm Conv} \Big( \overline{\text{\rm Sing$C$}} \cup \big(\pl C \cap \pl C_1 \big) \cup \big(\pl C \cap \pl C_2 \big) \Big).
$$
\end{corollary}

As applied to problem \eqref{f1}, we obtain the following statement (compare with Theorem \ref{T brfk}).

\begin{corollary}\label{cor4}
Let the function $g$ be continuous. Then there exists a solution $u$ to problem \eqref{f1}  possessing the following property. If there is an open set $\UUU \subset \Om$ such that

(a) $u \in C^1(\UUU)$,

(b) $u(x) < M$ for all $x \in \UUU$,

 then the surface graph$(u\rfloor_\UUU)$ does not contain extreme points of the subgraph of $u$.
\end{corollary}

The following statement concerns the more specific Newton's problem.

\begin{corollary}\label{cor3}
Let the boundary $\pl\Om$ be regular. Then there is a solution $u$ to generalized Newton's problem such that $C_u = \{ (x,y,z) :\, (x,y) \in \Om,\ 0 \le z \le u(x,y) \}$ satisfies the equation
\beq\label{ConvSing}
C_u = \text{\rm Conv} (\overline{\text{\rm Sing$C_u$}}).
\eeq
\end{corollary}

It follows that the boundary of the body $C_u$ is composed of line segments with the endpoints in the closure of the set of singular points. Yet, we cannot exclude the possibility that the set of singular points is dense in graph$(u)$, and therefore, all line segments contained in graph$(u)$ degenerate to singletons.

\begin{proof}
According to Corollary \ref{cor2}, $C_u = \text{\rm Conv} \Big( \overline{\text{\rm Sing$C_u$}} \cup \big(\pl C_u \cap \pl C_1 \big) \cup \big(\pl C_u \cap \pl C_2 \big) \Big),$ hence
$$
\text{Ext}C_u \subset \overline{\text{\rm Sing$C_u$}} \cup \big( \pl C_u \cap \pl C_1 \big) \cup \big(\pl C_u \cap \pl C_2 \big) $$
$$
\subset \overline{\text{\rm Sing$C_u$}} \cup \big( \text{Ext}C_u \cap \pl C_u \cap \pl C_1 \big) \cup \big( \text{Ext}C_u \cap \pl C_u \cap \pl C_2 \big).
$$
Recall that in generalized Newton's problem, $C_1 = \Om \times \{ 0 \}$ and $C_2 = \Om \times [0,\, M]$. Any extreme point of $C_u$ that lies in $\pl C_u \cap \pl C_1$, also belongs to the circumference $\pl\Om \times \{ 0 \}$, and therefore, to {\text{\rm Sing$C_u$}}. On the other hand, using property P$_5$, one sees that any extreme point of $C_u$ that lies in $\pl C_u \cap \pl C_2$, also belongs to the union of the circumference $\pl\Om \times \{ 0 \}$ and the boundary of the top level set $\pl \{ (x,y) : u(x,y) = M \}$. Both these sets, in turn, belong to {\text{\rm Sing$C_u$}}. Thus, one has
$$
\text{Ext}C_u \cap C_u \cap \pl C_1 \subset \text{\rm Sing$C_u$} \qquad \text{and} \qquad \text{Ext}C_u \cap C_u \cap \pl C_2 \subset \text{\rm Sing$C_u$},
$$
hence by the Krein-Milman theorem,
$$
C_u = \text{Conv}\big(\overline{\text{Ext}C_u}\big) $$
$$
\subset \text{Conv}\Big( \overline{\text{\rm Sing$C_u$}} \cup \big( \text{Ext}C_u \cap \pl C_u \cap \pl C_1 \big) \cup \big( \text{Ext}C_u \cap \pl C_u \cap \pl C_2 \big) \Big) = \text{\rm Conv} (\overline{\text{\rm Sing$C_u$}}).$$
The inverse inclusion $\text{\rm Conv} (\overline{\text{\rm Sing$C_u$}}) \subset C_u$ is obvious.
\end{proof}

   \begin{zam}
If the numerical observations described in Remark \ref{zam numerics} are true, and thus, there are only one or two distinct (up to a rotation) solutions to generalized Newton's problem with $\Om = \{ x_1^2 + x_2^2 \le 1 \}$, then each solution $u$ satisfies \eqref{ConvSing}.

In fact, numerical study seems to indicate that the optimal body is indeed the convex hull of the union of several curves composed of singular points: the circumference $\pl\Om \times \{ 0 \}$, the boundary of a regular polygon in the horizontal plane $\{ z = M \}$, and several convex curves 
joining each vertex of the polygon with a point of the circumference.
   \end{zam}

\section{Proof of Theorem \ref{T brfk}}\label{sec Tbrfk}

Assume the contrary, that is, there is a point $x_0 \in \UUU$ such that $D^2u(x_0) > 0$. Changing if necessary the orthogonal system of coordinates, one can assume that $x_0 = (0,0)$. For $\del > 0$ and $c > 0$ sufficiently small one has $D^2u(x) \ge c$ and $u''_{x_1x_1}(x) \le -c$ when $|x| \le \del$, and additionally, the circle $|x| \le \del$ is contained in $\UUU$.

Take a $C^2$ function $h : \RRR^2 \to \RRR$ equal to zero outside the circle $|x| \le \del$. For $|t|$ sufficiently small, $D^2\big(u(x)+th(x)\big) > 0$ and $u''_{x_1x_1}(x) + th''_{x_1x_1}(x) < 0$, and therefore, the function $u + th$ is concave. Besides, taking $|t|$ sufficiently small, one can ensure that $0 < u(x) + th(x) < M$ for all $x$.

Since $u$ minimizes the functional $\FFF$, we have
$$
\frac{d^2}{dt^2}\Big\rfloor_{t=0} \FFF[u + th] = \frac 12 \int\!\!\!\int_{\RRR^2} \nabla h(x)^T D^2g(\nabla u(x)) \nabla h(x)\, dx_1 dx_2
\ge 0
$$
(we represent the gradient as a row vector, $\nabla h = (h_x, h_y)$).

Now taking $h(x) = \phi(x/\ve)$, where $0 < \ve < 1$ and $\phi$ is a $C^2$ function vanishing outside the circle $|x| \le \del$, and making the change of variable $x = \ve y$, one obtains
$$
\int\!\!\!\int_{\RRR^2} \nabla\phi(y) D^2g(\nabla u(\ve y)) \nabla\phi(y)^T dy_1 dy_2 \ge 0.
$$

Passing to the limit $\ve \to 0$ one gets
$$
\int\!\!\!\int_{\RRR^2} \nabla\phi(y) D^2g(\nabla u(0)) \nabla\phi(y)^T dy_1 dy_2 \ge 0.
$$

Take one more change of variables $y = \Lam\chi$, where $\Lam$ is an orthogonal matrix with $\det\Lam = 1$ diagonalizing the matrix $D^2g(\nabla u(0))$, that is,
$$
\Lam^T D^2g(\nabla u(0)) \Lam = \left[\! \begin{array}{cc} a & 0\\ 0 & -b \end{array} \!\right] \quad \text{with} \ \ b > 0.
$$
Denoting $\psi(\chi) = \phi(\Lam\chi)$ and taking into account that $\nabla\psi(\chi) = \nabla\phi(\Lam\chi) \Lam$, one comes to the inequality
$$
\int\!\!\!\int_{\RRR^2} \nabla\psi(\chi) \left[\! \begin{array}{cc} a & 0\\ 0 & -b \end{array} \!\right] \nabla\psi(\chi)^T d\chi_1 d\chi_2
$$ \beq\label{2int}
= a \int\!\!\!\int_{\RRR^2} (\psi_{\chi_1}'(\chi))^2\, d\chi_1 d\chi_2 - b \int\!\!\!\int_{\RRR^2} (\psi_{\chi_2}'(\chi))^2\, d\chi_1 d\chi_2 \ge 0.
\eeq

Now let $\psi(\chi) = \psi(\chi, \tau) = \gam(\chi_1) \gam(\chi_2) \sin(\chi_2/\tau)$, where $\gam : \RRR \to \RRR$ is a smooth function vanishing outside a small neighborhood of 0 and $\tau \ne 0$. One easily checks that the former integral is bounded for all $\tau$,
$$
\int\!\!\!\int_{\RRR^2} (\psi_{\chi_1}'(\chi))^2\, d\chi_1 d\chi_2 \le \int \gam'^2(\chi_1)\, d\chi_1 \int \gam^2(\chi_2)\, d\chi_2,
$$
whereas the latter one goes to infinity as $\tau \to 0$,
$$
\int\!\!\!\int_{\RRR^2} (\psi_{\chi_2}'(\chi))^2\, d\chi_1 d\chi_2 = \frac{1}{\tau^2} \int \gam^2(\chi_1)\, d\chi_1 \int \gam^2(\chi_2) \cos^2(\chi_2/\tau)\, d\chi_2 + O(1/\tau)
$$
$$
\to +\infty \quad \text{as} \ \ \tau \to 0.
$$
It follows that the left hand side in \eqref{2int} is negative for $|\tau|$ sufficiently small. The contradiction finishes the proof.

\section{Proof of Theorem \ref{T main}}\label{sec Tmain}

The main idea of the proof consists in the procedure which is called {\it stretching the nose} and is described in Lemma \ref{l stretching}.

Namely, suppose that the body $C$ is a solution to problem \eqref{problem} but does not satisfy \eqref{eq property}. Take a point $\xi \in \text{Ext}C \setminus (\pl C_1 \cup \pl C_2 \cup \overline{{\text{\rm Sing$C$}}})$, then choose a point $O$ outside $C$ sufficiently close to $\xi$, and define a family of bodies $C(s)$, $0 \le s \le 1$ with the endpoints at $C(0) = C$ and $C(1) = \text{Conv}(C \cup \{ O \})$. The corresponding family of measures $\nu_{\pl C(s)}$ is a line segment. We prove that all bodies $C(s)$ are solutions to problem \eqref{problem}. This is the main point in the proof of Theorem \ref{T main}.

We also prove that if the set $\text{\rm Ext}C \setminus (\pl C_1 \cup \pl C_2 \cup \overline{{\text{\rm Sing$C$}}})$ is not empty, then it does not have isolated points, and therefore, is infinite (Lemma \ref{l isolated} and Corollary \ref{cor isolated}). Using this fact, in Lemma \ref{l cones} we find an infinite sequence $\xi_1,\, \xi_2, \ldots$ of points dense in this set  and a sequence of points $O_1,\, O_2, \ldots$ outside $C$ (each point $O_j$ is sufficiently close to $\xi_j$) so as the "noses" $\text{Conv}(C \cup \{ O_j \}) \setminus C$ are mutually disjoint and the body $\widehat C = \text{Conv}(C \cup \{ O_1,\, O_2, \ldots \})$ satisfies \eqref{eq property}. We define a family of intermediate convex bodies $C(\vec s)$, $\vec s \in [0,\, 1]^\infty$, with $C(\vec 0) = C$ and $C(\vec 1) = \widehat C$, such that $C(\underbrace{0, \ldots, 0,\, 1}_{j}, 0, \ldots) = \text{Conv}(C \cup \{ O_j \})$ and the corresponding family of surface area measures is linear, and prove that all bodies of the family are solutions to problem \eqref{problem}.

Denote by  $B_r(a)$ the open ball with radius $r$ and with the center at $a$. Let $C$ be a convex body.

\begin{lemma}\label{l isolated}
The set $\text{\rm Ext}C \setminus \overline{{\text{\rm Sing$C$}}}$ does not have isolated points.
\end{lemma}

\begin{proof}
Assume the contrary and let $\xi$ be an isolated point of $\text{\rm Ext}C \setminus \overline{{\text{\rm Sing$C$}}}$; then for some $\ve > 0$, the punctured neighborhood $B_\ve(\xi) \setminus \{ \xi \}$ does not intersect $\text{\rm Ext}C \cup \overline{{\text{\rm Sing$C$}}}$.

By Minkowski's Theorem,  $C = \text{Conv(Ext$C$)}$. Further, we have
$$
\text{Ext}C \subset \big( C \setminus B_\ve(\xi) \big) \cup \{ \xi \} \subset \text{Conv}\big( C \setminus B_\ve(\xi) \big) \cup \{ \xi \},
$$
hence
$$
C = \text{Conv(Ext$C$)} \subset \text{Conv} \Big[\text{Conv}\big( C \setminus B_\ve(\xi) \big) \cup \{ \xi \}\Big],
$$
and therefore, we have the equality
\beq\label{eqeq}
C = \text{Conv} \Big[\text{Conv}\big( C \setminus B_\ve(\xi) \big) \cup \{ \xi \}\Big].
\eeq
Since the point $\xi$ lies outside $\text{Conv}\big( C \setminus B_\ve(\xi) \big)$, it is a singular point of the convex body in the right hand side of \eqref{eqeq}. This contradicts the assumption that $\xi$ is not a singular point of $\pl C$.
\end{proof}

Consider three convex bodies $C,\, C_1,$ and $C_2$. Taking into account that the sets $\pl C_1$ and $\pl C_2$ are closed, we immediately obtain the following corollary of Lemma \ref{l isolated}.

\begin{corollary}\label{cor isolated}
The set $\text{\rm Ext}C \setminus (\pl C_1 \cup \pl C_2 \cup \overline{{\text{\rm Sing$C$}}})$ does not have isolated points.
\end{corollary}

The following Lemmas \ref{l kolpaki} and \ref{l O} are auxiliary; they will be used in the proof of Lemmas \ref{l cones} and \ref{l stretching}.

\begin{lemma}\label{l kolpaki}
Let $\{ O_1,\, O_2, \ldots \}$ be a finite or countable set of points outside $C$ such that for all $i \ne j$ the intersection of the open line interval $(O_i,\, O_j)$ with $C$ is not empty. Denote
$$\tilde C = \text{\rm Conv}(C \cup \{ O_1,\, O_2, \ldots \}).$$
Then

(a)
$$\tilde C = \cup_i \text{\rm Conv}(C \cup \{ O_i \}).$$

(b) If the set of points $O_i$ is finite, $\xi \in \text{\rm Ext}C \setminus \overline{\text{\rm Sing}C}$, and for all $i$ the intersection of the interval $(O_i,\, \xi)$ with $C$ is not empty, then $\xi \in \text{\rm Ext}\tilde C\setminus \overline{\text{\rm Sing}\tilde C}$.

(c)
$$
\{ O_1,\, O_2, \ldots \} \subset \text{\rm Ext}\tilde C \subset \text{\rm Ext}C \cup \{ O_1,\, O_2, \ldots \}.$$

(d) Assume, additionally, that for all $i$ and all $\xi \in \text{\rm Sing}C$ the intersection of the interval $(O_i,\, \xi)$ with $C$ is not empty. Then
$$\text{\rm Sing}C \cup \{ O_1,\, O_2, \ldots \} \subset \text{\rm Sing}\tilde C .$$
\end{lemma}

\begin{proof}
(a) 
Take a point $x \in \tilde C$. One needs to prove that for some $j$, $x \in \text{Conv}(C \cup \{ O_j \}).$ Two cases are possible: either $x \in C$, or $x$ is a convex combination
\beq\label{eq lcomb}x = \lam x_0 + \sum_1^m \lam_i O_i,
\eeq
where $x_0 \in C$, $\lam > 0$, $\lam_i > 0$ for all $i$, and $\lam + \sum_1^m \lam_i = 1$. In the former case there is nothing to prove. In the latter case assume, without loss of generality, that the linear combination in \eqref{eq lcomb} is the shortest one, that is, $m$ cannot be made smaller. Let us show that $m = 1$.

Indeed, suppose that $m \ge 2$. Since the segment $(O_1,\, O_2)$ intersects $C$, we have $\mu O_1 + (1-\mu) O_2 = \hat x \in C$ for some $0 < \mu < 1$. Assume without loss of generality that $\lam_1/\lam_2 \ge \mu/(1 - \mu)$; then
$$
\lam_1 O_1 + \lam_2 O_2 = \frac{\lam_2}{1 - \mu}\, \hat x + \lam_2 \Big( \frac{\lam_1}{\lam_2} - \frac{\mu}{1 - \mu} \Big) O_1,
$$
and the convex combination in \eqref{eq lcomb} can be shortened,
$$
x = \tilde\lam \tilde x_0 + \tilde\lam_1 O_1 + \sum_{i=3}^m \lam_i O_i,
$$
where
$$
\tilde\lam = \lam + \frac{\lam_2}{1 - \mu}, \quad \tilde x_0 = \frac{\lam(1 - \mu)}{\lam(1 - \mu) + \lam_2}\, x_0 +  \frac{\lam_2}{\lam(1 - \mu) + \lam_2}\, \hat x \in C, \quad \tilde\lam_1 = \lam_2 \Big( \frac{\lam_1}{\lam_2} - \frac{\mu}{1 - \mu} \Big).
$$
(If $m = 2$, the sum $\sum_3^m$ equals zero.) Note that $\lam x_0 + \frac{\lam_2}{1 - \mu}\, \hat x \in C$.

This contradiction shows that $m = 1$, that is, for some $j$,
$$x = \lam x_0 + \lam_j O_j \in \text{Conv}(C \cup \{ O_j \}) \quad \text{with} \ \, \lam > 0,\, \ \lam_j > 0,\, \ \lam + \lam_j = 1.
$$
Claim (a) is proved.

(b) Assume that $\xi$ is not an extreme point of $\tilde C$. This means that $\xi$ is an interior point of a line segment $[\xi_1,\, \xi_2] \subset \tilde C$. We are going to prove that each of the semiopen intervals $[\xi_1,\, \xi)$ and $[\xi_2,\, \xi)$ contains a point of $C$, and therefore, $\xi$ is not an extreme point of $C$, in contradiction with the hypothesis that $\xi \in \text{Ext}C$.

Suppose that $\xi_1 \not\in C$. Since the point $\xi_1$ is in $\tilde C$, by claim (a), for some $i$ we have $\xi_1 \in [O_i ,\, x_1)$, where $x_1 \in C$. On the other hand, by the hypothesis of the lemma, a point $x \in (O_i,\, \xi)$ lies in $C$.

If the triangle $\xi O_i x_1$ is non-degenerate then the segments $[x,\, x_1]$ and $(\xi,\, \xi_1)$ intersect at a point $\xi_1'$. Since $x$ and $x_1$ lie in $C$,\, $\xi_1'$ also belongs to $C$.

If, otherwise, the points $\xi,\, O_i,\, x_1$ are collinear then Conv$(\xi, x, x_1)$ is a segment on the line $O_i\xi$ and is contained in $C$ (since the points $\xi, x, x_1$ lie in $C$). The point $\xi_1 (\not\in C)$ lies on this line between the segment and the point $O_i$; hence $\xi_1 \in [O_i,\, \xi)$. The same inclusion holds for $x$; it follows that $x \in (\xi,\, \xi_1).$

Thus, in any case a point of the segment $[\xi_1,\, \xi)$ (either $\xi_1$, or $\xi_1'$, or $x$) belongs to $C$. The same argument holds for the segment $[\xi_2,\, \xi)$. Hence we have $\xi \not\in \text{\rm Ext}C$. This contradiction proves that $\xi \in \text{\rm Ext}\tilde C$.

It remains to prove that $\xi \not\in \overline{\text{\rm Sing}\tilde C}$. Indeed, assume the contrary; then $\xi$ is the limit of a sequence $\xi_i \in \text{\rm Sing}\tilde C$. Since the set of points $O_j$ is finite, there is a value $j$ such that infinitely many points $\xi_i$ are contained in Conv$(C \cup \{ O_j \})$. Additionally, for $i$ sufficiently large, $\xi_i$ do not coincide with $O_j$.

There is a neighborhood of $\xi$ that does not contain singular points of $\pl C$. If a regular point of $\pl C$ belongs to $\pl\tilde C$, then it is also a regular point of $\pl\tilde C$. It follows that for $i$ sufficiently large, $\xi_i$ are not contained in $\pl C$.

Thus, without loss of generality one can assume that all points $\xi_i$ lie in Conv$(C \cup \{ O_j \}) \setminus (C \cup \{ O_j \})$, and therefore, for some $x_i \in C$,\, $\xi_i \in (x_i,\, O_j)$. Taking if necessary a subsequence, we assume that $x_i$ converge to a certain $x \in C$, and therefore, $\xi \in [x,\, O_j)$.
By the hypothesis of the lemma, there is a point $x' \in C$ contained in $(\xi,\, O_j)$. We have $\xi \in [x,\, x') \subset C$. Since $\xi$ is an extreme point of $C$, we conclude that $\xi = x$.

Each plane of support to $\tilde C$ at $\xi_i$ is also a plane of support to $\tilde C$ at each point of the segment $[x_i,\, O_j]$, and in particular, at $x_i$. It is also a plane of support to $C$ at $x_i$. Since $\xi_i \in \text{\rm Sing}\tilde C$, there are more than one such plane, and therefore, $x_i \in \text{Sing}C$. It follows that $\xi = \lim_{i\to\infty} x_i \in \overline{\text{Sing}C}$. The obtained contradiction proves claim (b).

(c) The set $\pl\tilde C$ is the union of (i) a part of the boundary $\pl C$, (ii) open segments of the form $(O_i,\, \xi)$, where $\xi \in \pl C$ and the segment lies in a plane of support to $C$ through $O_i$, and (iii) the points $O_i$.

If a point $\xi \in \text{Ext}\tilde C$ belongs to $\pl C$ then it is also an extreme point of $C$. Open segments of the form $(O_i,\, \xi)$ obviously do not contain extreme points of $\tilde C$, and all points $O_i$ are extreme points of $\tilde C$. Claim (c) is proved.

(d) Obviously, each point $O_i$ is a singular point of $\pl\tilde C$.

Take a point $\xi \in \text{Sing}C$ and take a plane of support to $C$ at $\xi$. Since for all $i$,\, $(O_i,\, \xi) \cap C \neq \emptyset$, we conclude that the body $C$ and all points $O_i$ lie in the same half-space bounded by the plane. It follows that this plane is also a plane of support for the body $\tilde C = \text{\rm Conv}(C \cup \{ O_1,\, O_2, \ldots \}).$

Thus, each plane of support to $C$ at $\xi$ is also a plane of support to $\tilde C$. Since such a plane is not unique, we conclude that $\xi \in \text{Sing}\tilde C$. Claim (d) is proved.
\end{proof}

In what follows, {\it dist} means the Euclidean distance between two points.

\begin{lemma}\label{l O}
Consider a point $\xi \in \text{\rm Ext}C \setminus \overline{\text{\rm Sing}C}$ and a closed set $A \subset C$ that does not contain $\xi$. Then for any $\ve > 0$ there exists a point $O$ outside $C$ such that

(a) dist$(\xi, O) < \ve$;

(b) for any $x \in A$, the intersection of the open segment $(O,\, x)$ with the interior of $C$ is not empty.
\end{lemma}

\begin{proof}
The convex hull Conv$(C \setminus B_\ve(\xi))$ does not contain $\xi$. Making if necessary $\ve$ sufficiently small, we can assume that the set $A \cup \overline{\text{\rm Sing}C}$ is contained in Conv$(C \setminus B_\ve(\xi))$.

Take a plane $\Pi$ that separates the point $\xi$ and the set Conv$(C \setminus B_\ve(\xi))$.  Let this plane be given by $\langle x,\, n \rangle = c$, with $\xi$ being contained in the half-space $\langle x,\, n \rangle > c$ and Conv$(C \setminus B_\ve(\xi))$, in the complementary half-space $\langle x,\, n \rangle < c$. Here and in what follows, $\langle \cdot\,, \cdot \rangle$ means the scalar product. See Fig.~\ref{fig13}.

      \begin{figure}[h]
\begin{picture}(0,175)
\scalebox{1}{
\rput(6.5,3.25){
\psdots(-0.24,-1.692) (-0.4,-1.02) 
\psline(-0.4,-1.02)(0,-2.7)
\rput(0.03,-1.6){\scalebox{0.9}{$O$}}
\rput(-0.4,-0.7){\scalebox{1}{$\xi$}}
\rput(2.7,-0.3){\scalebox{1}{$\Pi$}}
\pscurve[linecolor=black,linewidth=0.8pt](-2.5,1.5)(-1.5,2.5)(1.5,2.5)(2.5,1.5)
\pscurve[linecolor=blue,linewidth=0.8pt](-2.5,1.5)(-1,-0.8)(1,-0.8)(2.5,1.5)
\psline[linecolor=black,linewidth=0.8pt](-2.9,0)(3,0)
\psline[linecolor=brown,linestyle=dashed](-1.77,0)(0,-2.7)(1.77,0)
\psline[linecolor=black,linewidth=0.8pt,linestyle=dashed](-2.9,-2.7)(3,-2.7)
\rput(2.7,-2.3){\scalebox{1}{$\Pi_1$}}
\rput(0,-3){\scalebox{1}{$\xi_1$}}
\rput(0,0.8){\scalebox{1.5}{$C$}}
\rput(0,2.4){\scalebox{1}{$A$}}
}
}
\end{picture}
\caption{The plane $\Pi$ separates $\xi$ and $A$. The set $\KKK$ is bounded above by $\Pi$ and below by the dashed line. The part of $C$ below $\Pi$ is contained in the $\ve$-neighborhood of $\xi$.}
\label{fig13}
\end{figure}
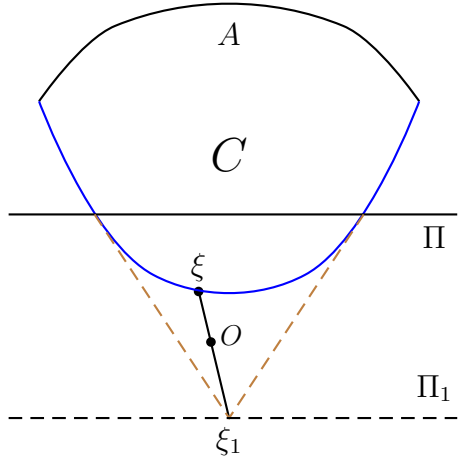

Draw the tangent planes to $C$ through all points of $\pl C \cap \Pi$ (which are regular); the intersection of the half-space $\langle x,\, n \rangle \ge c$ and all closed half-spaces bounded by these planes and containing $C$ is a convex set containing $\xi$. Let it be denoted by $\KKK$.

Let us show that $\KKK \setminus C$ is not empty. If $\KKK$ is unbounded, this is obvious. If $\KKK$ is bounded, draw the plane of support to $\KKK$ with the outward normal $n$ and denote it by $\Pi_1$. Thus, $\KKK$ is contained between the planes $\Pi$ and $\Pi_1$; see Fig.~\ref{fig13}.

Take a point $\xi_1$ in the intersection $\Pi_1 \cap \pl\KKK$. There is at least one more plane of support to $\KKK$ through a point of $\pl C \cap \Pi$ that contains $\xi_1$. It follows that $\xi_1$ is a singular point of $\pl\KKK$. Hence it does not belong to $C$, since otherwise it is also a singular point of $\pl C$. Thus, $\xi_1 \in \KKK \setminus C$.

Take a point $\xi'$ in the interior of $\KKK \setminus C$.
Draw the line segment $[\xi, \xi']$ and find a point $O$ on it that lies outside $C$ and belongs to $B_\ve(\xi)$.
Thus, condition (a) is satisfied, due to the choice of $O$.

Take a point $x \in A$. The point $x$ lies in the intersection of closed half-spaces bounded by the tangent planes to $C$ through all points of $\pl C \cap \Pi$ and containing $C$, and $O$ lies in the intersection of the corresponding open half-spaces. It follows that the point of intersection of the interval $(O,\, x)$ with the plane $\Pi$ lies in the interior of the planar set $C \cap \Pi$, and therefore, belongs to the interior of $C$. Thus, condition (b) is also satisfied.
\end{proof}

Assume that we are given three convex bodies $C_1 \subset C \subset C_2$.

\begin{lemma}\label{l cones}
Suppose that the set $E := \text{\rm Ext}C \setminus (\pl C_1 \cup \pl C_2 \cup \overline{{\text{\rm Sing$C$}}})$ is not empty and choose a point $\hat\xi \in E$. Then there exists an infinite sequence of points $\{ O_j,\, j \in \JJJ \}$ in $C_2 \setminus C$ such that

(a) for all $i \ne j$ and all $\xi \in C_1 \cup \overline{{\text{\rm Sing$C$}}} \cup \{ \hat\xi \}$, the intersections of the open segments $(O_i,\, O_j)$ and $(O_i,\, \xi)$ with the interior of $C$ are not empty;

(b) for $\widehat C = \text{\rm Conv}(C \cup \{ O_j,\, j \in \JJJ \})$ one has
$$
\text{\rm Ext}\widehat C \subset \pl C_1 \cup \pl C_2 \cup \overline{{\text{\rm Sing$\widehat C$}}}.
$$
\end{lemma}

\begin{proof}
Take a sequence of positive values $\ve_n$ converging to 0.
Choose a finite sequence of open sets (for example, open balls) $D_1, \ldots, D_{j_1}$ in $\RRR^3$, each set of diameter less than $\ve_1$, such that the union of the sets contains $E$. Next we define a finite sequence of open sets $D_{j_1 + 1}, \ldots, D_{j_2}$, each set of diameter less than $\ve_2$, such that their union contains $E$. Continuing this process, we obtain infinite sequences of integers $0 = j_0 < j_1 < j_2 < \ldots$ and sets $D_1,\, D_2, \ldots$ such that for each $n \ge 1$, the diameter of each of the domains $D_{j_{n-1} + 1}, \ldots, D_{j_n}$ is less than $\ve_n$ and
$$E \subset \bigcup_{i=j_{n-1} + 1}^{j_n} D_i.
$$

We are going to define inductively an infinite set of natural numbers $\JJJ \subset \NNN$ and a sequence of points $O_j$, $j \in \JJJ$, in $C_2 \setminus C$ satisfying condition (a). Denote $\{ 1, \ldots, m \}' := \JJJ \cap \{ 1, \ldots, m \}$. Define the sets
$$
\widehat C^m := \text{Conv}\big( C \cup \big\{ O_j,\, j \in \{ 1, \ldots, m \}' \big\} \big) \ \ \text{and} \ \
E_m := \text{\rm Ext}\widehat C^m \setminus (\pl C_1 \cup \pl C_2 \cup \overline{{\text{\rm Sing$\widehat C^m$}}})
$$
(in particular, $\widehat C^0 = C$ and $E_0 = E$); we additionally require that the sets $E_m$ contain $\hat\xi$ and are nested, that is, for $m_1 \le m_2$ we have $E_{m_2} \subset E_{m_1} \subset E$.

For $m = 0$ the set of points $O_i$ is empty, and therefore, condition (a) is trivially satisfied, and $\hat\xi \in E = E_0$. Now suppose that for a certain integer $m \ge 0$, the set $\{ 1, \ldots, m \}'$ is defined, the points $O_j$, $j \in \{ 1, \ldots, m \}'$ in $C_2 \setminus C$ satisfying condition (a) are chosen, and the inclusions $\hat\xi \in E_m \subset \ldots \subset E_1 \subset E_0 = E$ take place. If $E_{m} \cap D_{m+1} = \emptyset$, let $m+1 \not\in \JJJ$. In this case the statement of induction for $m+1$ is trivially satisfied.

If, otherwise, the set $E_m \cap D_{m+1}$ is not empty, let $m+1 \in \JJJ$ and take a point $\xi_{m+1}$ from this set distinct from $\hat\xi$. (Such a point exists, since by Corollary \ref{cor isolated}, the set $E_m$ does not have isolated points, and therefore, $E_m \cap D_{m+1}$ is not a singleton.)
Let $n$ be such that $m + 1 \in \{ j_{n-1} + 1, \ldots, j_n \}$. Using Lemma \ref{l O}, take a point $O_{m+1}$ in $C_2 \setminus \widehat C^m$ such that (a) dist$(\xi_{m+1}, O_{m+1}) < \ve_n$; (b) for all $\xi \in C_1 \cup \overline{{\text{\rm Sing$\widehat C^m$}}} \cup \{ \hat\xi \}$, the intersections of the open segment $(O_{m+1},\, \xi)$ with the interior of $C$ is not empty.

By the hypothesis of induction, for $i \in \{ 1, \ldots, m \}'$ and $\xi \in \text{Sing}C$, the intersection of the interval $(O_i,\, \xi)$ with the interior of $C$ is not empty. Hence by claim (d) of Lemma \ref{l kolpaki},
$$
\text{Sing}C \cup \big\{ O_i,\, i \in \{ 1, \ldots, m \}' \big\} \subset \text{Sing}\widehat C^m.
$$
It follows that for any $\xi \in C_1 \cup \overline{{\text{\rm Sing$C$}}} \cup \{ \hat\xi \}$, the intersection of the open segment $(O_{m+1},\, \xi)$ with the interior of $C$ is not empty, and for any $i \in \{ 1, \ldots, m \}'$, the intersection of $(O_{m+1},\, O_i)$ with the interior of $C$ is not empty. Thus, condition (a) is satisfied for the extended sequence of points $O_i,\, i \in \{ 1, \ldots, m+1 \}' = \{ 1, \ldots, m \}' \cup \{ m+1 \}$.

By claim (c) of Lemma \ref{l kolpaki}, $\text{Ext}\widehat C^{m+1} \subset \text{Ext}\widehat C^{m} \cup \{ O_{m+1} \}$, and by claim (d) of the same lemma, $\text{Sing}\widehat C^{m+1} \supset \text{Sing}\widehat C^{m} \cup \{ O_{m+1} \}$. It follows that
$$
E_{m+1} = \text{Ext}\widehat C^{m+1} \setminus (\pl C_1 \cup \pl C_2 \cup \overline{\text{Sing}\widehat C^{m+1}}) \subset
\text{Ext}\widehat C^{m} \setminus (\pl C_1 \cup \pl C_2 \cup \overline{\text{Sing}\widehat C^{m}}) = E_m.
$$
Further, since $\hat\xi \in E_m \subset \text{Ext}\widehat C^{m} \setminus \overline{\text{Sing}\widehat C^{m}}$ and the intersection of the interval $(O_{m+1},\, \hat\xi)$ with $C$ is non-empty, making use of claim (b) of Lemma \ref{l kolpaki} we conclude that $\hat\xi \in \text{Ext}\widehat C^{m+1} \setminus \overline{\text{Sing}\widehat C^{m+1}}$. It follows that $\hat\xi \in E_{m+1}$. The statement of induction is completely proved for $m+1$.

     In each subsequence $\{ j_{n-1} + 1, \ldots, j_n \}$ there is a number $m$ such that $D_m$ contains $\hat\xi$, and therefore, $E \cap D_m$ is not empty. It follows that $m \in \JJJ$; hence the parameter set $\JJJ$ is infinite.


We have proved that the sequence of points $O_j,\, j \in \JJJ$ satisfies claim (a) of Lemma \ref{l cones}.

Consider the set $\widehat C = \text{\rm Conv}(C \cup \{ O_j,\, j \in \JJJ \})$. By claims (c) and (d) of Lemma \ref{l kolpaki}, for any $m$ we have
\beqo\label{eq inc}
\text{Ext}\widehat C \subset \text{Ext}\widehat C^m \cup \{ O_j,\, j (\in \JJJ) \ge m+1 \} \quad \text{and} \quad
\text{Sing}\widehat C \supset \text{Sing}\widehat C^m \cup \{ O_j,\, j (\in \JJJ) \ge m+1 \},
\eeqo
hence
$$
E_{\infty} := \text{Ext}\widehat C \setminus (\pl C_1 \cup \pl C_2 \cup \overline{\text{Sing}\widehat C}) \subset
\text{Ext}\widehat C^m \setminus (\pl C_1 \cup \pl C_2 \cup \overline{\text{Sing}\widehat C^m}) = E_m \subset E.
$$
Let us show that $E_\infty$ is empty.

Assume the contrary and take $\xi \in E_\infty$. For any natural $n$ there is $m_n \in \{ j_{n-1} + 1, \ldots, j_n \}$ such that $D_{m_n}$ contains $\xi$. It follows that the set $E_{m_n-1} \cap D_{m_n} \supset E_\infty \cap D_{m_n} \ni \xi$ is non-empty, and therefore, $m_n \in \JJJ$. For the points $\xi_{m_n} \in D_{m_n}$ and $O_{m_n}$ chosen above in the proof we have
$$
\text{dist}(O_{m_n}, \xi) \le \text{dist}(O_{m_n}, \xi_{m_n}) + \text{dist}(\xi_{m_n}, \xi) < 2\ve_n.
$$
It follows that the sequence $O_{m_n}$ converges to $\xi$ as $n \to \infty$. Since by claim (d) of Lemma \ref{l kolpaki}, $\{ O_j,\, j \in J \} \subset \text{Sing}\widehat C$, we have $\xi \in \overline{\text{Sing}\widehat C}$, and therefore, $\xi \not\in E_\infty$. The obtained contradiction proves claim (b) of Lemma \ref{l cones}.
\end{proof}

The method we use in the following lemma can be called  {\it stretching the nose}.

Let $O$ be a point outside $C$. For $0 \le s \le 1$ define the set
\beq\label{eq stretching}
C(s) = \bigcup\limits_{\sqrt{1-s} \le \lam \le 1} \big( \lam C + (1 - \lam) O \big).
\eeq
In particular, $C(0) = C$ and $C(1) = \text{Conv}(C \cup \{ O \})$.

It is easy to see that $C(s)$ is a convex body.

The sets $C(0)$, $C(1)$, and $C(s)$ with $0 < s < 1$ are depicted in Figs.~\ref{fig2}(a), \ref{fig2}(b), and \ref{fig31}(a), respectively.

\begin{lemma}\label{l stretching}
Let $C$ be a solution to problem \eqref{problem} and let a point $O \in C_2 \setminus C$ be such that for any $\xi \in C_1 \cup \overline{{\text{\rm Sing$C$}}}$, the intersection of the open segments $(O,\, \xi)$ with the interior of $C$ is not empty. Then

(a) all convex bodies $C(s),\, 0 \le s \le 1$ given by \eqref{eq stretching} are also solutions to problem \eqref{problem};

(b) the measures $\nu_{\pl C(s)}$,\, $s \in [0,\, 1]$ form a linear segment: $\nu_{\pl C(s)} = \nu_{\pl C} + s \nu_0$, where $\nu_0$ is a signed measure on $S^2$.
\end{lemma}

This lemma is the main point in the proof of Theorem \ref{T main}. In turn, the main point in the proof of this lemma is extension of the family of admissible convex bodies $C(s)$ to negative values of $s$ and the statement that the composite function $F(\pl C(s))$ is linear for $s \in [0,\, 1]$ and differentiable at $s = 0$.

\begin{proof}
The sets $C(s)$ can be defined in another way.
Draw all the rays with vertex at $O$ that intersect $C$. The union of these rays is a closed convex cone. Denote by $A$ and $A'$ the initial (closer to $O$) and the final points of intersection of a generic ray with $C$. If the ray is tangent  then its intersection with $C$ is the line segment $[A, A']$ (which may degenerate to a point if $A = A'$). Otherwise, the intersection is the 2-point set $\{ A,\, A' \}$.

Denote by $C_-$ the union of the segments $OA'$ of all rays, and by $\pl_+ C$ the union of the corresponding points $A'$.
Denote by $V$ the surface composed of the segments $[O, A']$ contained in the tangent rays. The boundary of $C_-$ is the union $\pl_+ C \cup V$.

For each ray $OA$, denote by $AA'$ the ray contained in $OA$ with the vertex at $A$.
Denote by $C_+$ the union of the rays $AA'$, and by $\pl_- C$ the union of the points $A$ corresponding to all rays and the segments $[A, A']$ contained in the tangent rays. The boundary of $C_+$ is the union of $\pl_- C$ and the rays with the vertices at the points $A'$ contained in the tangent rays $OA'$.

       \begin{figure}[h]
\begin{picture}(0,120)
\rput(2,2.5){
\scalebox{0.6}{
\psline[linewidth=1.3pt]
(-3.4,0.6)(0,-3)(2.5,0)(3,0.6)
\rput(0.03,-3.4){\scalebox{1.5}{$O$}}
\rput(-3.45,-0.9){\scalebox{1.5}{$A_1 =A_1'$}}
\rput(2.35,-0.78){\scalebox{1.5}{$A_2$}}
\rput(3,0){\scalebox{1.5}{$A_2'$}}
\rput(0,1.9){\scalebox{1.5}{$\pl_+C$}}
\rput(-0.1,-1.95){\scalebox{1.4}{$\pl_-C$}}
\rput(1.3,-2.1){\scalebox{1.4}{$V$}}
\rput(-3,-3){\scalebox{1.3}{(a)}}
\psecurve[linecolor=blue,linewidth=1.4pt,fillstyle=solid,fillcolor=lightgray](1.5,-0.25)(2.5,0)(2.5,1)(0,1.5)(-2.5,0)(0.5,-1.5)(2,-0.6)(2.33,0)
\psdots[dotsize=4pt](2,-0.6) (2.5,0)(-2.24,-0.63)
\rput(0.1,0){\scalebox{3}{$C$}}
}}
\rput(6.5,2.5){
\scalebox{0.6}{
\pspolygon[linewidth=0pt,fillstyle=solid,fillcolor=lightgray](2,-0.6) (0,-3)(-2.24,-0.63)
\psline[linewidth=1.3pt]
(-3.4,0.6)(0,-3)(2.5,0)(3,0.6)
\rput(0.03,-3.4){\scalebox{1.5}{$O$}}
\rput(-0.1,1.9){\scalebox{1.5}{$\pl_+C$}}
\psecurve[linecolor=blue,linewidth=1.4pt,fillstyle=solid,fillcolor=lightgray,linewidth=1pt](1.5,-0.25)(2.5,0)(2.5,1)(0,1.5)(-2.5,0)(0.5,-1.5)(2,-0.6)(2.33,0)
\psdots[dotsize=4pt](0.99,1.56) (0.33,-1.5)
\rput(-0.4,-0.1){\scalebox{3}{$C_-$}}
\psline(0,-3)(1.2,2.5)
\rput(0.67,-1.05){\scalebox{1.5}{$A$}}
\rput(1.42,1.92){\scalebox{1.5}{$A'$}}
\rput(-3,-3){\scalebox{1.5}{(b)}}
}}
\rput(11.5,2.5){
\scalebox{0.6}{
\pspolygon[linewidth=0pt,linecolor=white,fillstyle=solid,fillcolor=lightgray]
(1,2.2)(0,2.5)(-1.25,2)(-2.75,2.25)(-4.0375,1.275)(-2.24,-0.63)(2.5,0)(3.625,1.35)(3,2.25)(2.25,2.5)
\psline[linewidth=1.3pt](-4.25,1.5)
(-3.4,0.6)(0,-3)(2.5,0)(3,0.6)(3.75,1.5)
\rput(0.03,-3.4){\scalebox{1.5}{$O$}}
\rput(-2.6,-0.8){\scalebox{1.5}{$A_1$}}
\rput(2.35,-0.78){\scalebox{1.5}{$A_2$}}
\rput(3,0){\scalebox{1.5}{$A_2'$}}
\rput(-0.1,-1.95){\scalebox{1.4}{$\pl_-C$}}
\psecurve[linecolor=blue,linewidth=1.4pt,fillstyle=solid,fillcolor=lightgray,linewidth=1pt](1.5,-0.25)(2.5,0)(2.5,1)(0,1.5)(-2.5,0)(0.5,-1.5)(2,-0.6)(2.33,0)
\psdots[dotsize=4pt](2,-0.6) (2.5,0)(-2.24,-0.63)
\rput(0.3,0.6){\scalebox{3}{$C_+$}}
\rput(-3,-3){\scalebox{1.5}{(c)}}
}}
\end{picture}
\caption{The sets $C$, $C_-$, and $C_+$ are shown in figures (a), (b), and (c), respectively.
The upper curve $A_1 A_2' $ is $\pl_+ C$, and the lower curve $A_1 A_2 A_2'$ is $\pl_- C$.}
\label{fig2}
\end{figure}
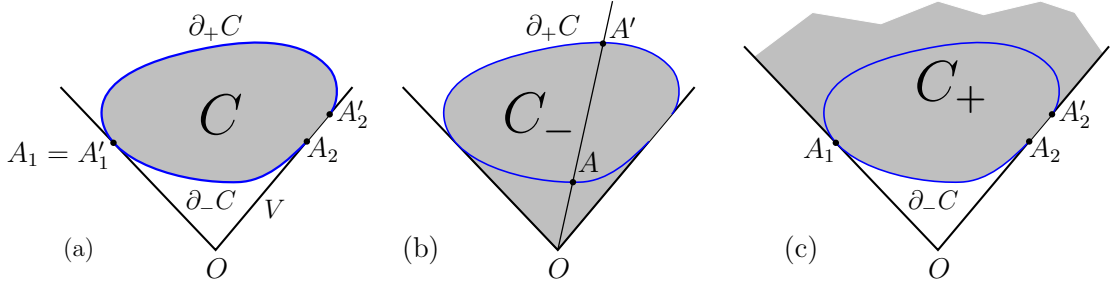

 We have $C = C_- \cap C_+$; see Fig.~\ref{fig2}.

Place the origin at the point $O$; then $tC$ designates the dilation of $C$ with the center at $O$ and the ratio $t$. The set $C(s)$ now takes the form (see Fig.~\ref{fig31})
\beq\label{C(s)}
C(s) = \left\{ \begin{array}{ll}
C_- \cap \sqrt{1-s}\, C_+, & \text{if} \ s < 1\\
C_-, & \text{if} \ s = 1. \end{array} \right.\eeq
In particular, $C(0) = C_- \cap C_+ = C$ and $C(1) = C_- = \text{Conv}(C \cup \{ O \})$.
Since $C_1 \subset C \subset C(s) \subset \text{Conv}(C \cup \{ O \}) \subset C_2$, the bodies $C(s)$, $0 \le s \le 1$ are admissible.

      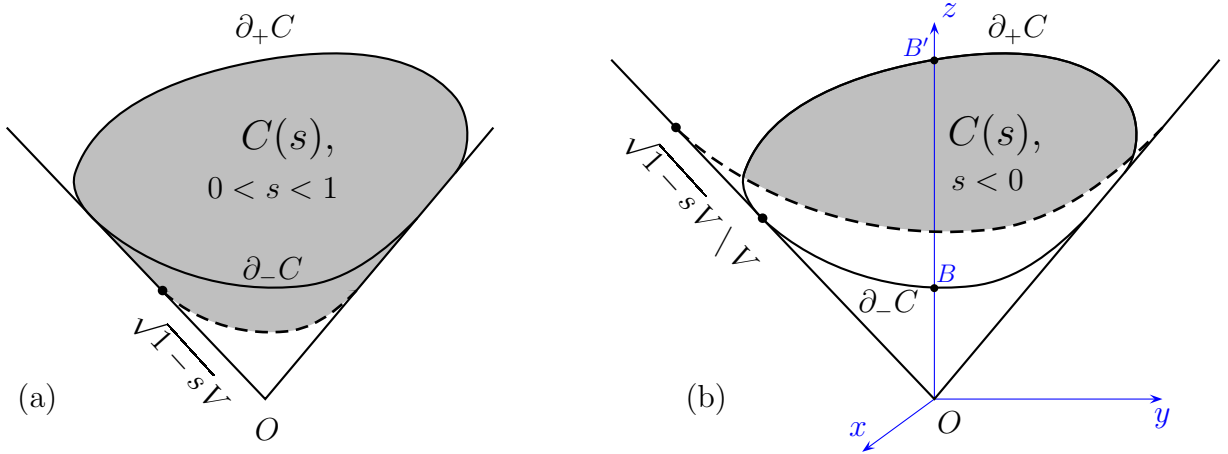
\begin{figure}[h]
\begin{picture}(0,170)
\rput(3.4,3.8){
\scalebox{1}{
\pspolygon[linewidth=0pt,fillstyle=solid,fillcolor=lightgray](-2.24,-0.63)(2,-0.6)(1.2,-1.56)(0,-1.7)(-1.344,-1.578)
\rput(0.03,-3.4){\scalebox{1}{$O$}}
\rput(0,1.9){\scalebox{1}{$\pl_+C$}}
\rput(-3,-3){\scalebox{1}{(a)}}
\psecurve[linecolor=black,linewidth=1pt,fillstyle=solid,fillcolor=lightgray,linestyle=dashed]
(-1.5,-1.2)(-1.344,-1.578)(0.3,-2.1)(1.2,-1.56)(1.398,-1.2)
\psecurve[linecolor=black,linewidth=0.8pt,fillstyle=solid,fillcolor=lightgray]
(1.5,-0.25)(2.5,0)(2.5,1)(0,1.5)(-2.5,0)(-2.24,-0.63)(0.5,-1.5)(2,-0.6)(2.33,0)
\rput(0.3,0.45){\scalebox{1.3}{$C(s),$}}
\rput(0.1,-0.2){\scalebox{1}{$0 < s < 1$}}
\rput(0.1,-1.3){\scalebox{1}{$\pl_-C$}}
\psline[linewidth=0.8pt]
(-3.4,0.6)(0,-3)(2.5,0)(3,0.6)
   \psdots(-1.35,-1.56)
   \rput{-48}(-1.1,-2.5){$\sqrt{1-s}\, V$}
}}
\rput(12.2,3.8){
\scalebox{1}{
\psecurve[linecolor=black,linewidth=0.8pt,fillstyle=solid,fillcolor=lightgray]
(1.5,-0.25)(2.5,0)(2.5,1)(0,1.5)(-2.5,0)(-2.24,-0.63)(0.5,-1.5)(2,-0.6)(2.33,0)
\pscurve[linecolor=white,linewidth=0pt,fillstyle=solid,fillcolor=white]
(0,-3)(-4.4,1.5)(-4.3,1.7)(-4,1.7)(-3.75,1.5)(-3.36,0.555)(0.75,-0.75)(3,0.6)(3.495,1.5)(3.7,1.7)(3.9,1.7)(4.1,1.3)(4,0)(0,-3)
     \psline[linecolor=blue,linewidth=0.4pt,arrows=->,arrowscale=1.5](0,-3)(0,2)
     \rput(0.2,2.15){\scalebox{1}{$\blue z$}}
          \psline[linecolor=blue,linewidth=0.4pt,arrows=->,arrowscale=1.5](0,-3)(3,-3)
     \rput(3,-3.25){\scalebox{1}{$\blue y$}}
          \psline[linecolor=blue,linewidth=0.4pt,arrows=->,arrowscale=1.5](0,-3)(-0.95,-3.7)
     \rput(-1,-3.4){\scalebox{1}{$\blue x$}}
\rput(0.2,-3.3){\scalebox{1}{$O$}}
\rput(1.1,1.9){\scalebox{1}{$\pl_+C$}}
\rput(-0.6,-1.75){\scalebox{1}{$\pl_-C$}}
\rput(0.16,-1.33){\scalebox{0.8}{$\blue B$}}
\rput(-0.24,1.67){\scalebox{0.8}{$\blue B'$}}
\rput(-3,-3){\scalebox{1}{(b)}}
\psdots[dotsize=3pt](0,-1.53)(0,1.49)
\psecurve[linecolor=black,linewidth=1pt,linestyle=dashed]
(-3.75,1.5)(-3.36,0.555)(0.75,-0.75)(3,0.6)(3.495,1.5)
\psecurve[linecolor=black,linewidth=0.8pt]
(1.5,-0.25)(2.5,0)(2.5,1)(0,1.5)(-2.5,0)(-2.24,-0.63)(0.5,-1.5)(2,-0.6)(2.33,0)
  \psdots(-2.26,-0.6)(-3.4,0.6)
\rput(0.8,0.5){\scalebox{1.3}{$C(s),$}}
\rput(0.7,-0.1){\scalebox{1}{$s < 0$}}
\psline[linewidth=0.8pt](-4.25,1.5)
(-3.4,0.6)(0,-3)(2.5,0)(3,0.6)(3.75,1.5)
   \rput{-48}(-3.2,-0.4){$\sqrt{1-s}\, V \setminus V$}
}}
\end{picture}
\caption{The set $C(s)$ (a) for $0 < s < 1$ and (b) for $s < 0$. The surface $\sqrt{1-s}\ \pl_- C$ is shown as dashed line.}
\label{fig31}
\end{figure}

Note that formula \eqref{C(s)} defines the body $C(s)$ also for the values $s < 0$.

All intervals $(O,\, A)$ and the intervals $(O,\, A')$ contained in the tangent rays do not intersect the interior of $C$, hence by the hypothesis of the lemma, no point of these intervals belongs to $C_1 \cup \overline{{\text{\rm Sing$C$}}}$. The set $\pl_-C$ belongs to the union of these intervals, therefore $\pl_-C$ does not intersect $C_1 \cup \overline{{\text{\rm Sing$C$}}}$. Since both sets, $\pl_-C$ and $C_1 \cup \overline{{\text{\rm Sing$C$}}}$, are compact, for $|s|$ sufficiently small the dilated set $\sqrt{1-s}\ \pl_-C$ also does not intersect $C_1 \cup \overline{{\text{\rm Sing$C$}}}$, and therefore in particular, $C_1 \subset C(s) \subset C_2$, that is, $C(s)$ is admissible.

Let us now study the composite function $F(\pl C(s))$. For $0 \le s \le 1$ this function is linear. Indeed, $\pl C(s)$ is composed of the surfaces $\pl_+ C$, $\sqrt{1-s}\, \pl_-C$, and $V \setminus \sqrt{1-s}\, V$ (note that $\sqrt{1-s}\, V \subset V$). The surface area measure of $\pl C(s)$ is
$\nu_{\pl C(s)} = \nu_{\pl_+C} + (1-s) \nu_{\pl_-C} + s \nu_V.$ It can be represented as
$$
\nu_{\pl C(s)} = \nu_{\pl C} + s \nu_0, \ \ s \in [0,\, 1], \quad \text{where} \ \ \nu_0 = \nu_V - \nu_{\pl_-C}.
$$
Thus, claim (b) of the lemma is proved.

For $0 \le s \le 1$ we have $F(\sqrt{1-s}\, \pl_-C) = (1 - s) F(\pl_-C)$ and $F(V \setminus \sqrt{1-s}\, V) = F(V) - (1-s) F(V) = s F(V)$, therefore
$$
F(\pl C(s)) = F(\pl_+ C) + (1-s) F(\pl_- C) + s F(V).
$$

Let us now show that the derivative $\frac{d}{ds}\big\rfloor_{s=0} F(\pl C(s))$ exists.

The calculation of the right derivative is straightforward,
$$
\frac{d}{ds}\Big\rfloor_{s=0^+} F(\pl C(s)) = \lim_{s\to 0^+} \frac{F(\pl C(s)) - F(\pl C)}{s} = F(V) - F(\pl_- C).
$$

For $s < 0$, the boundary of the convex body $C(s)$ is composed of parts of the surfaces $\pl_+ C$ and $\sqrt{1-s}\ \pl_- C$. Namely,
$$
\pl C(s) = (\pl_+ C \cap \sqrt{1-s}\, C_+) \cup (\sqrt{1-s}\ \pl_- C \cap C_-),
$$
and the complementary parts of these surfaces, $\pl_+ C \setminus \sqrt{1-s}\, C_+$ and $\sqrt{1-s}\, \pl_- C \setminus C_-$, do not take part of the boundary. Therefore we have
$$
F(\pl C(s)) = F(\pl_+ C) + (1-s) F(\pl_- C) - \big[ F(\pl_+ C \setminus \sqrt{1-s}\, C_+) + F(\sqrt{1-s}\ \pl_- C \setminus C_-) \big]
$$
$$
= F(\pl_+ C) + (1-s) F(\pl_- C) + s F(V)
$$
$$
- \big[ s F(V) + F(\pl_+ C \setminus \sqrt{1-s}\, C_+) + F(\sqrt{1-s}\ \pl_- C \setminus C_-) \big].
$$
Therefore, the left derivative (if exists) equals
$$
\frac{d}{ds}\Big\rfloor_{s=0^-} F(\pl C(s)) = \lim_{s\to 0^-} \frac{F(\pl C(s)) - F(\pl C)}{s} = F(V) - F(\pl_- C) - \lim_{s\to 0^+} \frac{R(s)}{s},
$$
where
$$
R(s) = s F(V) + F(\pl_+ C \setminus \sqrt{1-s}\, C_+) + F(\sqrt{1-s}\, \pl_- C \setminus C_-)
$$
$$
= F(\pl_+ C \setminus \sqrt{1-s}\, C_+) + F(\sqrt{1-s}\, \pl_- C \setminus C_-) - F(\sqrt{1-s}\, V \setminus V).
$$

Let us prove that $R(s) = o(s)$ as $s \to 0^-$; it will follow that the derivative $\frac{d}{ds}\big\rfloor_{s=0} F(\pl C(s))$ exists and is equal to $F(V) - F(\pl_- C)$.

Draw a straight line through $O$ intersecting the interior of $C$. Let $B$ and $B'$ be the points of intersection of this line with $\pl C$, so as the open segment $OB$ is outside $C$.

Introduce an ortogonal coordinate system with the coordinates $x$, $y$, $z$ so as the origin is at $O$ and the $z$-axis coincides with the axis $OB$; see Fig.~\ref{fig31}\,(b). Let $D_s$,\, $D_s^+$,\, $D_s^-$, $s < 0$, be the orthogonal projections of $\sqrt{1-s}\, V \setminus V$,\, $\pl_+ C \setminus \sqrt{1-s}\, C_+$,\, $\sqrt{1-s}\, \pl_- C \setminus C_-$, respectively, on the $xy$-plane. The area of $D_s$ equals $-ks$, where $k$ is the area of the corresponding projection of $V$. For $|s|$ sufficiently small, the domains $D_s^+$ and $D_s^-$ have disjoint interiors and $D_s^+ \cup D_s^- = D_s$, hence
$$
\text{Area}(D_s^+) + \text{Area}(D_s^-) = \text{Area}(D_s) = -ks.
$$

Denote by $n(x,y) = (n_1(x,y),\, n_2(x,y),\, n_3(x,y))$ the outward normal to $\sqrt{1-s}\, V \setminus V$ at the pre-image of $(x,y) \in D_s$ under the projection. Similarly, let $n^+(x,y) = (n_1^+(x,y),\, n_2^+(x,y),\, n_3^+(x,y))$ and $n^-(x,y) = (n_1^-(x,y),\, n_2^-(x,y),\, n_3^-(x,y))$ be the outward normals to $\pl_+ C \setminus \sqrt{1-s}\, C_+$ and $\sqrt{1-s}\, \pl_- C \setminus C_-$, respectively. The third components of these vectors, $n_3(x,y),\, n_3^+(x,y),$ and $n_3^-(x,y)$, are negative for $|s|$ sufficiently small.

The function $n(x,y)$ is continuous in $D_s$ and is constant in the radial direction. The function $n^+(x,y)$ coincides with $n(x,y)$ on the inner boundary of $D_s$; in other words, for any $(x,y) \in D_s^+$ there exists $0 < c \le 1$ such that $(cx,cy) \in D_s^+$ and $n(x,y) = n(cx,cy) = n^+(cx,cy)$. The function $n^+(x,y)$ is continuous (and therefore uniformly continuous) in the closure of $D_{s}$ for $|s|$ sufficiently small. Thus, for any $\ve > 0$ there exists $\del > 0$ such that for all $-\del < s < 0$, for all $(x,y) \in D_s^+$, and for a suitable positive $c = c(x,y) \le 1$ we have  $(cx, cy) \in D_s^+$,\, $n(cx,cy) = n^+(cx,cy)$, and $|n^+(x,y) - n^+(cx,cy)| < \ve$; it follows that for any $(x,y) \in $ $D_s^+$,\, $|n^+(x,y) - n(x,y)| < \ve$.

A similar reasoning holds for the function $n^-(x,y)$ with $(x,y) \in D_s^-$. As a result we have
$$
\sup_{(x,y) \in D_s^+} |n^+(x,y) - n(x,y)| + \sup_{(x,y) \in D_s^-} |n^-(x,y) - n(x,y)| \to 0 \quad \text{as} \ \, s \to 0^-.
$$

For $|s|$ sufficiently small the function $p(n) = \frac{f(n)}{|n_3|}$
is well defined, and therefore is uniformly continuous, in the closure of the set
$\{ n(x,y) : (x,y) \in D_s \} \cup \{ n^+(x,y) : (x,y) \in D_s^+ \} \cup \{ n^-(x,y) : (x,y) \in D_s^- \} \subset S^2,$
hence
$$
\sup_{(x,y) \in D_s^+}|p(n(x,y)) - p(n^+(x,y))| +  \sup_{(x,y) \in D_s^-}|p(n(x,y)) - p(n^-(x,y))| =: \al(s) \to 0 \ \ \text{as} \ \, s \to 0^-.
$$

We have
$$
R(s) 
= \int_{\pl_+ C \setminus \sqrt{1-s}\, C_+} f(n_\xi)\, d\xi + \int_{\sqrt{1-s}\, \pl_- C \setminus C_-} f(n_\xi)\, d\xi - \int_{\sqrt{1-s}\, V \setminus V} f(n_\xi)\, d\xi.
$$
Making the change of variable $\xi \rightsquigarrow x, y$ in these integrals and taking into account that $d\xi = \frac{dx\, dy}{|n_3^+(x,y)|}$,\, $d\xi = \frac{dx\, dy}{|n_3^-(x,y)|}$, and $d\xi = \frac{dx\, dy}{|n_3(x,y)|}$ in the first, second, and third integrals, respectively, we get
$$
R(s) = \int_{D_s^+} p(n^+(x,y))\, dx\, dy + \int_{D_s^-} p(n^-(x,y))\, dx\, dy - \int_{D_s} p(n(x,y))\, dx\, dy
$$
$$
= \int_{D_s^+} \left( p(n^+(x,y)) - p(n(x,y)) \right) dx\, dy + \int_{D_s^-} \left( p(n^-(x,y)) - p(n(x,y)) \right) dx\, dy,
$$
and so,
$$
|R(s)| \le \int_{D_s^+} |( p(n^+(x,y)) - p(n(x,y)) )|\, dx\, dy + \int_{D_s^-} |( p(n^-(x,y)) - p(n(x,y)) )|\, dx\, dy $$
$$\le \al(s) \, k|s|= o(s) \ \, \text{as} \ \, s \to 0^-.
$$

It follows that there exists the derivative
$$
\frac{d}{ds}\Big\rfloor_{s=0} F(\pl C(s)) = F(V) - F(\pl_- C).
$$
Since $F(\pl C(s))$ takes the minimal value at $s = 0$, we have $\frac{d}{ds}\big\rfloor_{s=0} F(\pl C(s)) = 0$, therefore $F(\pl C(s))$ is constant for $0 \le s \le 1$. Thus, all bodies $C(s)$, $0 \le s \le 1$ are solutions to problem \eqref{problem}. Claim (a) of the lemma is also proved.
\end{proof}

Let us finish the proof of the theorem.

Let $C$ be a solution to problem \eqref{problem}. Assuming that the set $\text{\rm Ext}C \setminus (\pl C_1 \cup \pl C_2 \cup \overline{\text{\rm Sing$C$}})$ is not empty, we use Lemma \ref{l cones} to obtain an infinite sequence of points $O_1,\, O_2, \ldots$ in $C_2 \setminus C$ such that
(a) for all $i \ne j$ and all $\xi \in \pl C_1 \cup \pl C_2 \cup \overline{{\text{\rm Sing$C$}}}$, the intersections of the open segments $(O_i,\, O_j)$ and $(O_i,\, \xi)$ with the interior of $C$ are not empty;
(b) for $\widehat C = \text{\rm Conv}(C \cup \{ O_1,\, O_2, \ldots \})$ holds
$\text{\rm Ext}\widehat C \subset \pl C_1 \cup \pl C_2 \cup \overline{{\text{\rm Sing$\widehat C$}}}.$

For any $\vec s = (s_1,\, s_2, \ldots ) \in [0,\, 1]^\infty$ denote
$$
C(\vec s) = \cup_i C_i(s_i), \ \ \text{where} \ \ C_i(s) = \bigcup\limits_{\sqrt{1-s} \le \lam \le 1} \big( \lam C + (1 - \lam) O_i \big).
$$

One has $C(\vec 0) = C$ and $C(\vec 1) = \widehat C = \cup_i \text{Conv}(C \cup \{ O_i \})$. Denoting $\vec e_i = (\underbrace{0, \ldots, 0,\, 1}_{i}, 0, \ldots)$, we have $C(\vec e_i) = \text{Conv}(C \cup \{ O_j \})$.
By claim (a) of Lemma \ref{l kolpaki}, $\widehat C = \text{Conv}(C \cup \{ O_1,\, O_2, \ldots \})$, and by claim (b) of Lemma \ref{l cones},
$\text{\rm Ext}\widehat C \subset \pl C_1 \cup \pl C_2 \cup \overline{{\text{\rm Sing$\widehat C$}}}.$

Each set $C(\vec s)$ is convex. This is proved in Appendix 1. We have $C_1 \subset C \subset C(\vec s) \subset C_2$, therefore $C(\vec s)$ belongs to the class of admissible bodies.

Fix $i$ and consider all rays with vertex at $O_i$ intersecting $C$. Denote by $A$ and $A'$ the initial (closer to $O_i$) and the final points of intersection of a generic ray of this kind with $C$. Let $\pl_i^- C$ be the union of the points $A$ corresponding to all rays and the segments $[A,\, A']$ contained in the rays tangent to $C$. Let $V_i$ be the union of all segments $[O_i,\, A']$ contained in the tangent rays.

The closed sets bounded by the subspaces $V_i$ and $\pl_i^- C$ are mutually disjoint. This is be proved in Appendix 2.

The boundary of $C$ is the disjoint union of all surfaces $\pl_i^- C$ and the remaining part of the boundary, $\pl C \setminus (\cup_i \pl_i^- C)$,
$$
\pl C = \cup_i \pl_i^- C  \bigcup \Big( \pl C \setminus (\cup_i \pl_i^- C) \Big).
$$

Denote by $T_i(k)$ the dilation with center $O_i$ and ratio $k$ and consider the convex body $C(s\vec e_i),\, 0 \le s \le 1$. Its boundary is the disjoint union
$$
\pl C(s\vec e_i) = \big( \pl C \setminus \pl_i^- C \big) \cup \Big( T_i(\sqrt{1-s})(\pl_i^- C) \cup \left( V_i \setminus T_i(\sqrt{1-s})(V_i) \right) \Big).
$$
Correspondingly, its surface measure is
$$
\nu_{\pl C(s\vec e_i)} = (\nu_{\pl C} - \nu_{\pl_i^- C}) + (\nu_{T_i(\sqrt{1-s})(\pl_i^- C)} + \nu_{V_i} - \nu_{T_i(\sqrt{1-s})(V_i)})
$$
$$
= (\nu_{\pl C} - \nu_{\pl_i^- C}) + \big( (1-s) \nu_{\pl_i^- C} + \nu_{V_i} - (1-s) \nu_{V_i} \big) = \nu_{\pl C} + s \nu_i,
$$
where $\nu_i = \nu_{V_i} - \nu_{\pl_i^- C}$, and
$$
F(\pl C(s\vec e_i)) = F( \nu_{\pl C}) + s (F(V_i) - F(\pl_i^- C).
$$
Since by Lemma \ref{l stretching}, every convex body $C(s\vec e_i)$ is a solution to problem \eqref{problem}, we have
\beq\label{equal}
F(V_i) - F(\pl_i^- C) = 0. \eeq

In general, the boundary of $C(\vec s),\, \vec s \in [0,\, 1]^\infty$ is the disjoint union of $\pl C \setminus (\cup_i \pl_i^- C)$ and the surfaces $T_i(\sqrt{1-s_i})(\pl_i^- C)$ and $V_i \setminus T_i(\sqrt{1-s_i})(V_i)$ for all values of $i$,
$$
\pl C(\vec s) = \bigcup_i  \Big( T_i(\sqrt{1-s_i})(\pl_i^- C) \cup \left( V_i \setminus T_i(\sqrt{1-s_i})(V_i) \right) \Big)
\bigcup \Big( \pl C \setminus (\cup_i \pl_i^- C) \Big);
$$
see Fig.~\ref{fig4} for the case when $\vec s = (\frac{1}{2},\, \frac{1}{2},\, 0,\, 0, \ldots)$.
     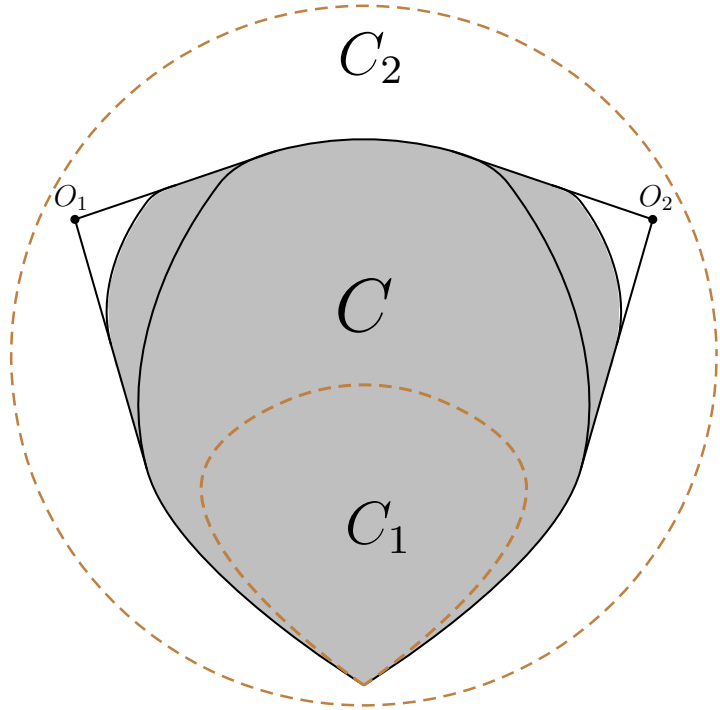
\begin{figure}[h]
\begin{picture}(0,260)
\scalebox{0.95}{
\rput(7.5,0.25){
  \psdots(-4,6.5) (4,6.5)
    \pscurve[linecolor=blue,linestyle=dashed,linewidth=0.8pt,fillstyle=solid,fillcolor=lightgray,linewidth=0pt,linecolor=white]
    (0,0)(-3,3)(-2,7)(-1.2,7.46)(1.2,7.44)(1.9,7.22)  (2.6,6.97)(2.84,6.87)(3,6.72)  (3.5,5.65)(3.5,4.75)(3.25,3.87)   (2.95,3)(0,0)
    \pscurve[linecolor=blue,linestyle=dashed,linewidth=0.8pt,fillstyle=solid,fillcolor=lightgray,linewidth=0pt,linecolor=white]
    (0,0)(3,3)(2,7)(1.2,7.46)(-1.2,7.44)(-1.9,7.22)  (-2.6,6.97)(-2.84,6.87)(-3,6.72)  (-3.5,5.65)(-3.5,4.75)(-3.25,3.87)   (-2.95,3)(0,0)
\pscurve[linecolor=black,linewidth=0.8pt,fillstyle=solid,fillcolor=lightgray](0,0)(-3,3)(-2,7)(-1.2,7.46)(1.2,7.46)(2,7)(3,3)(0,0)
\pscurve[linecolor=brown,linestyle=dashed,linewidth=1.2pt](0,0)(-2.25,2.8)(-1,4)(1,4)(2.25,2.8)(0,0)
  \psecurve[linecolor=black,linewidth=0.8pt](1.4,6.98)(2.6,6.98)(3,6.75)(3.5,4.75)(2,3.25)
  \psecurve[linecolor=black,linewidth=0.8pt](-1.4,6.98)(-2.6,6.98)(-3,6.75)(-3.5,4.75)(-2,3.25)
  \psline(4,6.5)(1.2,7.46)
  \psline(4,6.5)(3,3)
    \psline(-4,6.5)(-1.2,7.46)
  \psline(-4,6.5)(-3,3)
   \rput(4.05,6.8){\scalebox{1}{$O_2$}}
   \rput(-4.05,6.8){\scalebox{1}{$O_1$}}
   \rput(0,5.3){\scalebox{2.5}{$C$}}
   \rput(0.2,2.2){\scalebox{2}{$C_1$}}
   \rput(0.1,8.75){\scalebox{2}{$C_2$}}
   \pscircle[linecolor=brown,linestyle=dashed,linewidth=1pt](0,4.6){4.9}
}
}
\end{picture}
\caption{The body $C(\vec s)$ with $\vec s = (\frac{1}{2},\, \frac{1}{2},\, 0,\, 0, \ldots)$ is shown in gray. The body $C(\vec 0)$ coincides with $C$. The body $C(\vec 1)$ is bounded by the closed line through $O_1$ and $O_2$. The bodies $C_1$ and $C_2$ are bounded by dashed lines.}
\label{fig4}
\end{figure}
Correspondingly, the surface measure is
$$
\nu_{\pl C(\vec s)} = \sum_i \Big( \nu_{T_i(\sqrt{1-s_i})(\pl_i^- C)} + \nu_{V_i} - \nu_{T_i(\sqrt{1-s_i})(V_i)} \Big) +
\Big( \nu_{\pl C} - \sum_i \nu_{\pl_i^- C} \Big)
$$
$$
= \sum_i \Big( (1-s_i) \nu_{\pl_i^- C} + \nu_{V_i} - (1-s_i) \nu_{V_i} \Big) + \Big( \nu_{\pl C} - \sum_i \nu_{\pl_i^- C} \Big)
= \nu_{\pl C} + \sum_i s_i (\nu_{V_i} - \nu_{\pl_i^- C}).
$$
Hence we have
$$
\nu_{\pl C(\vec s)} = \nu_{\pl C} + \sum_i s_i \nu_i \quad \text{with} \quad \nu_i = \nu_{V_i} - \nu_{\pl_i^- C}
$$
and
$$
F(\pl C(\vec s)) = F(\pl C) + \sum_i s_i (F(V_i) - F(\pl_i^- C)).
$$
Using \eqref{equal}, one obtains $F(\pl C(\vec s)) = F(\pl C)$, that is, for every $\vec s$ the convex body $C(\vec s)$ is a solution to problem \eqref{problem}.

\section*{Acknowledgements}

This work is supported by CIDMA through FCT (Funda\c{c}\~ao para a Ci\^encia e a Tecnologia), reference UIDB/04106/2020.

\section*{Appendix 1}

Since all sets $C_i(s_i)$ are convex, it suffices to show that for all $i \ne j$, $x_i \in C_i(s_i)$ and $x_j \in C_j(s_j)$, the segment $[x_i,\, x_j]$ belongs to $C_i(s_i) \cup C_j(s_j)$.

We have
\beq\label{for}
x_i = \mu_1 \hat x_i + (1-\mu_1) O_i \quad \text{and} \quad x_j = \mu_2 \hat x_j + (1-\mu_2) O_j
\eeq
for some $\hat x_i,\, \hat x_j \in C$, $\sqrt{1-s_i} \le \mu_1 \le 1$,\, $\sqrt{1-s_j} \le \mu_2 \le 1$. If $\mu_1 = 1$ or $\mu_2 = 1$ then the segment $[x_i,\, x_j]$ belongs to $C_j(s_j)$ or $C_i(s_i)$, respectively. If, otherwise, both $\mu_1$ and $\mu_2$ are not equal to 1 then for some $0 < \lam < 1$ the point $x_0 = \lam O_i + (1-\lam) O_j$ lies in $C$.

Take the point $\bar x = \tilde\lam x_i + (1 - \tilde\lam) x_j$ with
$$
\tilde\lam = \frac{\frac{\lam}{1-\mu_1}} {\frac{\lam}{1-\mu_1} + \frac{1-\lam}{1-\mu_2}},
\qquad
1 - \tilde\lam = \frac{\frac{1-\lam}{1-\mu_2}} {\frac{\lam}{1-\mu_1} + \frac{1-\lam}{1-\mu_2}}.
$$
Using formula \eqref{for}, one sees that $\bar x$ is a convex combination of the points $x_0$, $\hat x_i$, $\hat x_j$, and therefore, lies in $C$.

Thus, the segment $[x_i,\, x_j]$ is divided by the point $\bar x$ into two parts $[x_i,\, \bar x]$ and $[\bar x,\, x_j]$, with $x_i \in C_i(s_i)$,\, $x_j \in C_j(s_j)$,\, $\bar x \in C$. Hence the former segment belongs to $C_i(s_i)$ and the latter one belongs to $C_j(s_j)$.

Thus, $C(\vec s)$ is convex.

\section*{Appendix 2}

Let us show that for $i \ne j$,

(a) the segments $[O_i,\, A_i]$ and $[O_j,\, A_j]$ contained, respectively, in rays from $O_i$ and $O_j$ are disjoint;

(b) the segments $[O_i,\, A'_i]$ and $[O_j,\, A'_j]$ contained in tangent rays from $O_i$ and $O_j$ are disjoint.

(a) Suppose that $A_i = A_j$. The points $O_i,\, O_j, \, A_i$ are not collinear, since otherwise, using that the interval $(O_i,\, O_j)$ contains an interior point of $C$, one sees that one of the of the open intervals $(O_i,\, A_i)$ and $(O_j,\, A_j)$ intersects $C$, which is impossible.

Since $A_i \in C$ and a point of $(O_i,\, O_j)$ is in the interior of $C$, we conclude that the intersection of $C$ with the plane $O_i O_j A_i$ is a planar convex body contained in the angle $O_i A_i O_j$, and both lines $O_i A_i$ and $O_j A_i$ are lines of support to this planar body. It follows that the plane $O_i O_j A_i$ is tangent to $C$. This contradicts to the condition that $(O_i,\, O_j)$ contains an interior point of $C$.

Let now $A_i \neq A_j$, and take a point of intersection $\xi$ of the segments $[O_i,\, A_i)$ and $[O_j,\, A_j)$. Using that $[A_i,\, A_j] \subset C$,\, $(A_i,\, \xi] \cap C = \emptyset$,\, $(A_j,\, \xi] \cap C = \emptyset$, one concludes that the points $A_i$, $A_j$, $\xi$ are not collinear.

A point of $(O_i,\, O_j)$ (let it be $O$ and the segment $[A_i,\, A_j]$ are contained in $C$. The point $\xi$ belongs to the planar triangle $O A_i A_j$, hence it also belongs to $C$, which is impossible.

(b) Let $\xi$ be the point of intersection of the segments $[O_i,\, A'_i]$ and $[O_j,\, A'_j]$ (they are not collinear, since otherwise the tangent line $O_j O_j$ does not intersect the interior of $C$). The point $\xi$ is interior for both segments $[O_i,\, A'_i]$ and $[O_j,\, A'_j]$, since otherwise we have either $\xi = A'_i$ or $\xi = A'_j$, and the plane through $\xi,\, A'_i,\, A'_j$ is tangent to $C$, in contradiction with the condition that the segment $[O_i,\, O_j]$ contains an interior point of $C$.

Again, for some $0 < \lam < 1$ the point $O := \lam O_i + (1-\lam) O_j$ lies in the interior of $C$. Besides, the segment $[A'_i,\, A'_j]$ is contained in $C$. The point $\xi$ is interior for the planar triangle $O A'_i A'_j$, hence it belongs to the interior of $C$. This contradicts the fact that the segments $[O_i,\, A'_i]$ and $[O_j,\, A'_j]$ do not intersect the interior of $C$.

\end{document}